\documentclass[amssymb,amsfonts,refcheck,11pt,verbatim,righttag]{amsart}

\usepackage{amssymb}

\usepackage{graphicx}

 \numberwithin{equation}{section}
\theoremstyle{plain}

\newtheorem{thm}{Theorem}[section]
\newtheorem{lem}[thm]{Lemma}
\newtheorem{pro}[thm]{Proposition}
\newtheorem{cor}[thm]{Corollary}
\newtheorem{ex}[thm]{Example}
\newtheorem{de}[thm]{Definition}
\newtheorem{rem}[thm]{Remark}

\def\eproof{\hfill $\Box$ \bigskip}

\def\R {{\Bbb R}}
\def\N {{\Bbb N}}

\def\M {{\mathcal M}}

\def\I {{\mathcal I}}
\def\A {{\mathcal A}}

\def\C {{\mathcal C}}
\def\D {{\mathcal D}}
\def\G {{\mathcal G}}

\def\ba{{\bf a}}
\def\bc{{\bf c}}

\def\bp{{\bf p}}
\def\bq{{\bf q}}

\begin{document}
\baselineskip 16pt
\title{Weighted thermodynamic formalism and applications.}

\author{Julien Barral}
\address{LAGA (UMR 7539), D\'epartement de Math\'ematiques, Institut Galil\'ee, Universit\'e Paris 13, 99 avenue Jean-Baptiste Cl\'ement , 93430  Villetaneuse, France}
\email{barral@math.univ-paris13.fr}
\author{De-Jun Feng}
\address{
Department of Mathematics\\
The Chinese University of Hong Kong\\
Shatin,  Hong Kong\\
} \email{djfeng@math.cuhk.edu.hk}


\date{}
\begin{abstract}

Let $(X,T)$ and $(Y,S)$ be two subshifts so that $Y$ is a factor of
$X$.
 For any asymptotically sub-additive potential $\Phi$ on $X$ and $\ba=(a,b)\in
\R^2$ with $a>0$, $b\geq 0$, we introduce the notions of
$\ba$-weighted topological pressure and $\ba$-weighted equilibrium
state of $\Phi$. We setup the weighted variational principle. In the
case that $X, Y$ are full shifts with one-block factor map, we prove
the uniqueness and Gibbs property of $\ba$-weighted equilibrium
states for almost additive potentials having the bounded distortion
properties. Extensions are given to the higher dimensional weighted
thermodynamic formalism. As an application, we conduct the
multifractal analysis for a new type of level sets associated with Birkhoff averages, as well as for weak Gibbs measures
associated with asymptotically additive potentials  on self-affine symbolic
spaces.
\end{abstract}

\maketitle
\section{Introduction}\label{intro}

The classical thermodynamic formalism developed by Sinai, Ruelle,
Bowen and Walters plays a fundamental role in statistical mechanics
and dynamical systems (see, e.g. \cite{Rue78, Wal82}). It adapts to
describe geometric properties of invariant sets and measures for
situations in which the statistics (box counting) carry all the
useful geometric information (e.g. the Hausdorff dimension of
conformal sets and measures \cite{Bow79, Rue82}, the topological entropy of level sets
of Birkhoff averages \cite{Bow73, Pes97, PeWi97}). However it seems not so
efficient when statistical and geometrical point of views reveal
different behaviors (e.g. the Hausdorff dimension of non-conformal
sets and measures). In this paper we
develop the so-called  weighted thermodynamic formalism, which may
provide a frame for which non-conformal geometry can be understood
through natural thermodynamical quantities. This is indeed the case for the dynamics of expanding diagonal endomorphisms of tori. For instance, let $m_1>m_2\ge 2$ be two integers,  let $K\subset \mathbb{T}^2$ be a self-affine Sierpinski carpet invariant by $T={\rm diag} (m_1,m_2)$, and $S$ denotes the map $y\mapsto m_2y(\mbox{mod } 1)$; let $\pi$ be the restriction to $K$ of the second coordinate projection. Let $\ba=(a,b):=(1/\log m_1, 1/\log m_2-1/\log m_1)$. Our starting point is to substitute the $\ba$-weighted entropy $h^\ba_\mu (T)=ah_\mu (T)+bh_{\mu\circ\pi^{-1}} (S)$ to the classical one in the variational definition of the topological pressure of any continuous potential $\phi$; this yields the ``$\ba$-weighted pressure'' $P^\ba(T,\phi)$ (in this setting, the Hausdorff dimension of $K$ obtained in \cite{McM84, Bed84, KePe96} is $P^\ba(T,0)$).
Then, we derive the uniqueness and the new Gibbs property for  $\ba$-weighted equilibrium states associated with any continuous potential $\phi$ satisfying the bounded distortion property, and prove for this case the differentiability of the $\ba$-weighted pressure function of $\phi$, namely $P^\ba(T,q\phi)$. This is used to establish a bridge between this weighted thermodynamic formalism and the Hausdorff dimension of invariant subsets of $K$, thanks to a fundamental result claiming that any invariant measure $\mu$ is the limit, in the weak-star topology, of a sequence of $\ba$-weighted equilibrium states whose $\ba$-weighted entropies converge to that of $\mu$. This property, as well as the Ledrappier-Young type formula $\dim_ H\nu= h^\ba_\nu (T)$  for any ergodic measure $\nu$ (see \cite{KePe96}), are exploited to find a sharp lower bound for the Hausdorff dimension of the set of generic points of any invariant measure $\mu$, which turns out to be equal to $h^\ba_\mu (T)$. It is also exploited to conduct the multifractal analysis of a new family of level sets associated with  the Birkhoff averages of $\phi$. There, the Hausdorff dimensions of level sets are expressed via the Legendre transform of $P^\ba(T,q\phi)$.


 In fact, our results hold in the more general framework for ``self-affine symbolic spaces'' and almost additive potentials. Before formulating them, we first give some definitions.

 We say that $(X,T)$ is a {\it topological dynamical system}
(TDS) if $X$ is a compact metric space and $T$ is a continuous map
from $X$ to $X$. Assume that $(X,T)$ and $(Y, S)$ are two TDSs such
that there is a continuous surjective map $\pi: X\to Y$ with $\pi
T=S\pi$, that is,  $Y$ is a {\it factor of $X$ with  factor map
$\pi$}. Let $\Phi=(\log \phi_n)_{n=1}^\infty$ be a sequence of
functions on $X$. We say that $\Phi$ is a {\it sub-additive
potential} and write $\Phi\in \C_{s}(X,T)$ if $\phi_n$ is
non-negative continuous for each $n$ and there exists a constant $c>0$ such that
$$
\phi_{n+m}(x)\leq c\phi_n(x)\phi_m(T^nx),\quad \forall \;x\in X, \; n,m\in \N.
$$
(we admit that $\phi_n$ takes the value zero).
More generally,   $\Phi=(\log \phi_n)_{n=1}^\infty$ is said to be an {\it asymptotically sub-additive potential} and  write $\Phi\in \C_{ass}(X,T)$ if for any $\varepsilon>0$, there exists a sub-additive potential $\Psi=(\log \psi_n)_{n=1}^\infty$ on $X$ such that
$$
\limsup_{n\to \infty}\frac{1}{n}\sup_{x\in X} |\log \phi_n(x)-\log \psi_n(x)|\leq \varepsilon,
$$
where we take the convention $\log 0-\log 0=0$.
Furthermore $\Phi$ is called an {\it asymptotically additive potential} and  write $\Phi\in \C_{asa}(X,T)$
if both $\Phi$ and $-\Phi$ are asymptotically sub-additive, where $-\Phi$ denotes $(\log (1/\phi_n))_{n=1}^\infty$.
In particular,  $\Phi$ is called {\it additive} if each $\phi_n$ is a continuous positive-valued function so that
$\phi_{n+m}(x)=\phi_n(x)\phi_m(T^nx)$ for all $x\in X$ and $m,n\in \N$; in this case, there is a
continuous real function $g$
on $X$ such that $\phi_n(x)=\exp(\sum_{i=0}^{n-1}g(T^ix))$ for each $n$.

Let $\Phi=(\log \phi_n)_{n=1}^\infty$ be an  asymptotically
sub-additive potential on $X$.
Let $\ba=(a,b)\in
\R^2$ so that $a>0$ and  $b\geq 0$. We introduce
\begin{equation}
\label{e-0.1}
P^{\ba}(T,\Phi)=\sup\{\Phi_*(\eta)+ah_\eta(T)+bh_{\eta\circ \pi^{-1}}(S):\; \eta\in \M(X,T)\},
\end{equation}
where $\M(X,T)$ denotes the collection of $T$-invariant probability
measures on $X$ endowed with the weak-star topology, $h_{\eta}(T)$
and $h_{\eta\circ \pi^{-1}}(S)$ denote the measure theoretic
entropies of $\eta$ and $\eta\circ\pi^{-1}$ (cf. \cite{Wal82}), and
$\Phi_*(\eta)$ is given by
\begin{equation}
\label{e-0.2}
 \Phi_*(\eta)=\lim_{n\to \infty}\frac{1}{n}\int \log
\phi_n(x)\; d\eta(x).
\end{equation}
By subadditivity, the limit in (\ref{e-0.2}) always exists (but may
take the value $-\infty$).  We call $P^{\ba}(T,\Phi)$ the {\it
$\ba$-weighted topological pressure of $\Phi$}. A measure $\eta\in
\M(X,T)$ is called an {\it $\ba$-weighted equilibrium state of
$\Phi$} if the supremum in (\ref{e-0.1}) is attained at $\eta$.

When $\ba=(1,0)$, we write $P^{\ba}(T,\Phi)$ simply as $P(T,\Phi)$ and call it
 the {\it  topological pressure of $\Phi$}. We remark that $P(T,\Phi)$ is a natural
 generalization of the classical topological pressure of additive
 functions, and it has been defined in an alternative way via separated sets or open covers
 in \cite{CFH08}.

Let $\nu\in \M(Y,S)$. We say that $\mu\in \M(X,T)$ is a {\it
conditional equilibrium state of $\Phi$ with respect to $\nu$} if
$\mu\circ \pi^{-1}=\nu$ and
\begin{equation}
\label{e-0.3}
\Phi_*(\mu)+h_\mu(T)-h_\nu(S)=\sup\{\Phi_*(\eta)+h_\eta(T)-h_\nu(S):\;\eta\in
\M(X,T),\; \eta\circ \pi^{-1}=\nu\}.
\end{equation}

 In the remainder part of this section,  we assume
that $X$ is a subshift over a finite alphabet $\A$, and $Y$  a
subshift over a finite alphabet $\D$ together with a one-block
factor map $\pi:\; X\to Y$ (see \S\ref{S-2.1} for the definitions).
Under this setting, the entropy function is upper semi-continuous
and hence the supremums in (\ref{e-0.1}) and (\ref{e-0.3}) are
attainable. For $I\in \A^n$, the $n$-th cylinder set $[I]$ in
$\A^\N$ is defined as
$$[I]=\{(x_i)_{i=1}^\infty\in \A^{\N}:\; x_1\ldots x_n=I\}.$$
Similarly for $J\in \D^n$, let $[J]$ denote the $n$-th cylinder set
in $\D^\N$. Our first result is the following.

\begin{thm}
\label{thm-0.1} Let $\ba=(a,b)\in \R^2$ so that $a>0$ and  $b\geq
0$. Let $\Phi=(\log \phi_n)_{n=1}^\infty$ be an asymptotically sub-additive potential on $X$, i.e. $\Phi \in \C_{ass}(X,T)$. Define a sequence $\Psi=(\log
\psi_n)_{n=1}^\infty$ of functions  on $Y$ by
\begin{equation*}\label{e-0.4}
\psi_n(y)=\sum_{I\in \A^n:\; [I]\cap \pi^{-1}(y)\neq \emptyset}
\sup_{x\in [I]\cap \pi^{-1}(y)}\phi_n(x)^{\frac 1a},\quad y\in Y.
\end{equation*}
Set $\frac{a}{a+b}\Psi=\left(\log\left(
\psi_n^{\frac{a}{a+b}}\right)\right)_{n=1}^\infty$. Then $\Psi$ and
$\frac{a}{a+b}\Psi$ are in $\C_{ass}(Y,S)$, moreover
\begin{equation}
\label{e-0.5}
\begin{split} P^\ba(T,\Phi)&=(a+b)
P\left(S,\frac{a}{a+b}\Psi\right)\\
&=\lim_{n\to \infty}\frac{a+b}{n}\log \sum_{J\in \D^n}\sup_{y\in
[J]\cap Y}\psi_n(y)^{\frac{a}{a+b}}.
\end{split}
\end{equation}
Furthermore, $\mu\in \M(X,T)$ is an $\ba$-weighted equilibrium state
of $\Phi$ if and only if  $\nu=\mu\circ \pi^{-1}$ is an equilibrium
state of $\frac{a}{a+b}\Psi$ and, $\mu$ is a conditional equilibrium
state of $\frac{1}{a}\Phi$ with respect to $\nu$, where
$\frac{1}{a}\Phi$ denotes the potential $(\log (\phi_n^{1/a}))_{n=1}^\infty$.

\end{thm}

Formula (\ref{e-0.5}) can be viewed as a kind of weighted
variational principle. To further study  weighted equilibrium
states, we shall put more assumptions on $\Phi$. We say that
$\Phi=(\log \phi_n)_{n=1}^\infty$ is {\it almost additive} if
$\phi_n$ is positive and continuous on $X$  for each $n$ and  there
is a constant $c>0$ such that
\begin{equation*}
\label{e-taa}
\frac{1}{c} \phi_n(x)\phi_m(T^nx)\leq \phi_{n+m}(x)\leq c
\phi_n(x)\phi_m(T^nx),\quad \forall \; x\in X,\; n,m\in \N.
\end{equation*}
For convenience, we denote by  $\C_{aa}(X,T)$ the collection of
almost additive potentials on $X$. Furthermore we say that $\Phi$
has the {\it bounded distortion property} if there exists a constant
$c>0$ such that
\begin{equation}\label{BDP}
\frac{1}{c}\phi_n(y)\leq \phi_n(x)\leq c\phi_n(y)\quad \mbox{
whenever $x,y\in X$ are in the same $n$-th cylinder}.
\end{equation}

Following \cite{BMP92}, a  full supported Borel probability  measure $\mu$ on $\A^\N$ is
called to be {\it quasi-Bernoulli} if there exists a constant $c>0$
such that
\begin{equation}\label{quasib}
c^{-1}\leq \frac{\mu(IJ)}{\mu(I)\mu(J)}\leq c,\qquad \forall
\; I, J\in \A^*:=\bigcup_{i=1}^\infty\A^n,
\end{equation}
here and afterwards,  we use $\mu(I)$ to denote $\mu([I])$ for $I\in \A^*$, if there is no confusion.
For two families $\{a_i\}_{i\in \I}$, $\{b_i\}_{i\in \I}$ of
non-negative numbers, we write $a_i\approx b_i$ if there exists
$c>0$ such that $(1/c)b_i\leq a_i\leq cb_i$ for all $i\in \I$.  Our
next result is the following.

\begin{thm}
\label{thm-0.2} Assume that $X=\A^\N$ and $Y=\D^\N$ are two full
shifts and $\pi:\; X\to Y$ is a one-block factor map. Let
$\ba=(a,b)\in \R^2$ so that $a>0$ and  $b\geq 0$. Let  $\Phi=(\log \phi_n)_{n=1}^\infty\in
\C_{aa}(X,T)$. Assume that $\Phi$ satisfies the bounded distortion
property. Then $\Phi$ has  a unique $\ba$-weighted equilibrium
state, denoted as $\mu$. The measure $\mu$ is quasi-Bernoulli and
has the following Gibbs property:
\begin{equation}
\label{e-2.2} \mu(I)\approx \exp\left(
\frac{-nP^{\ba}(T,\Phi)}{a+b} \right) \frac{\phi(I)^{1/a}} {\psi(\pi
I)^{b/(a+b)}},\quad I\in \A^n,\; n\in \N,
\end{equation}
where $$\phi(I):=\sup_{x\in [I]}\phi_n(x)\mbox{ for }I\in \A^n\quad
\mbox{and}\quad \psi(J):=\sum_{I\in \A^n:\; \pi
I=J}\phi(I)^{1/a}\mbox{ for }J\in \D^n.$$ Furthermore for
$\nu:=\mu\circ \pi^{-1}$, we have
\begin{equation}
\label{e-2.2'} \nu(J)\approx \exp\left(
\frac{-nP^{\ba}(T,\Phi)}{a+b} \right) \psi(J)^{\frac{a}{a+b}},\quad
J\in \D^n,\; n\in \N
\end{equation}
and
\begin{equation}
\label{e-2.3} \mu(I)^a\nu(\pi I)^b\approx
\phi(I)\exp(-nP^{\ba}(T,\Phi)),\quad I\in \A^n,\; n\in \N.
\end{equation}
\end{thm}

A probability measure $\mu$ (not necessarily to be $T$-invariant) on
$X$ is called  an {\it $\ba$-weighted Gibbs measure}, if there
exists $\Phi\in \C_{aa}(X,T)$ satisfying the bounded distortion
property so that  \eqref{e-2.2} holds for $\mu$. Clearly, any
$\ba$-weighted Gibbs measure is quasi-Bernoulli. As an application
of Theorem \ref{thm-0.2}, we have the following result regarding the
regularity property of $P^\ba(T,\cdot)$.

\begin{thm}
\label{thm-0.3} Under the  assumptions of Theorem \ref{thm-0.2}, let
$\Phi_1,\ldots, \Phi_d\in \C_{aa}(X,T)$  satisfy the bounded
distortion property. Then the map $Q:\R^d \to \R$ defined as
$$\bq=(q_1,\ldots, q_d)\mapsto P^{\ba}\left(T,\sum_{i=1}^dq_i\Phi_i\right),$$
is
$C^1$ over $\R^d$ with
$$\nabla Q(q_1,\ldots, q_d)=((\Phi_1)_*(\mu_\bq),\ldots,
(\Phi_d)_*(\mu_\bq)),$$
 where
$\nabla$ denotes the gradient and $\mu_{\bq}$ is the unique
$\ba$-weighted equilibrium state of $\sum_{i=1}^dq_i\Phi_i$.

\end{thm}

Using Theorems \ref{thm-0.2} and \ref{thm-0.3}, we derive the
following two results, which play key roles in the multifractal
analysis on self-affine symbolic spaces,  and  are  of independent
interest.

\begin{thm}
\label{thm-0.4} Assume that $X=\A^\N$ and $Y=\D^\N$ are two full
shifts and $\pi:\; X\to Y$ is a one-block factor map. Let
$\ba=(a,b)\in \R^2$ so that $a>0$ and  $b\geq 0$.  Then for each
fully supported measure $\eta\in \M(X,T)$ and each $n\in \N$, there
is a unique measure $\mu=\mu(\ba, \eta,n)$ in $\M(X,T)$ attaining
the following supremum
$$
\sup\{ah_\mu(T)+bh_{\mu\circ \pi^{-1}}(S):\;
\mu(I)=\eta(I) \mbox{ for all } \omega\in \A^n\}.
$$
Furthermore $\mu(\ba, \eta,n)$ is the $\ba$-weighted equilibrium
state of certain $\Phi\in \C_{aa}(X,T)$ with the bounded distortion
property, and hence $\mu(\ba, \eta,n)$ is a  fully supported
quasi-Bernoulli measure.
\end{thm}

\begin{thm}
\label{thm-0.5} Under the condition of Theorem \ref{thm-0.4}, for
any $\eta\in \M(X,T)$, there exists a sequence of $\ba$-weighted
Gibbs measures $(\mu_n)_{n=1}^\infty \subset \M(X,T)$ converging  to
$\eta$ in the weak-star topology such that
\begin{equation*}
\label{e-3.1}ah_{\mu_n}(T)+bh_{\mu_n\circ \pi^{-1}}(S)\geq
ah_\eta(T)+bh_{\eta\circ \pi^{-1}}(S).
\end{equation*}
Furthermore,
\begin{equation*}
\label{e-0.}
 \lim_{n\to \infty}ah_{\mu_n}(T)+bh_{\mu_n\circ
\pi^{-1}}(S)= ah_\eta(T)+bh_{\eta\circ \pi^{-1}}(S).
\end{equation*}
\end{thm}
\bigskip

\begin{rem}\label{simultapprox}
{\rm
If we take $\ba=(1,1)$, due to the upper semi-continuity of the entropy, for any $\mu\in\M(X,T)$, Theorem~\ref{thm-0.5} yields a sequence of quasi-Bernoulli measures $(\mu_n)_{n=1}^\infty$ which converges to $\mu$ in the weak-star topology,  such that we have both $\lim_{n\to\infty} h_{\mu_n}(T)= h_{\mu}(T)$ and  $\lim_{n\to\infty} h_{\mu_n\circ\pi^{-1}}(S)=  h_{\mu\circ\pi^{-1}}(S)$. Moreover, one can deduce from Theorem~\ref{thm-0.2} that for any $\ba=(a,b)$ with $a>0$ and $b\geq 0$,  each  invariant quasi-Bernoulli measure is the $\ba$-weighted equilibrium state of some almost additive  potential satisfying the bounded distortion property.
}
\end{rem}
Now we present our results about the multifractal analysis on
self-affine symbolic spaces. In the remainder part of the section,
we always assume that $X=\A^\N$ and $Y=\D^\N$ are two full shifts
and $\pi:\; X\to Y$ is a one-block factor map. Endow $X$ with a
metric $d_\ba$ as follows:
\begin{equation*}
d_{\ba}(x,y)=\max\Big(e^{-|x\land y|/a},\; e^{-|\pi x\land \pi
y|/(a+b)}\Big ),
\end{equation*}
where $|x\land y|=\inf\{k\geq 1:\; x_k\neq y_k\}-1$ and $|\pi x\land
\pi y|=\inf\{k\geq 1:\; \pi x_k\neq \pi y_k\}-1$. The space $X$,
endowed with the metric $d_\ba$,  is called a {\it self-affine full
shift}. Indeed if
\begin{equation}\label{planarset}
e^{-\frac{1}{a}} \cdot \sup _{j\in \D} \# \pi^{-1}\{j\}<1 \quad\mbox { and }\quad e^{-\frac{1}{a+b}}\cdot \# \D<1,
\end{equation}
the space $(X, d_{\ba})$ is Lipschitz equivalent to a planar self-affine set  generated by a linear iterated function system $\{S_i\}_{i\in \A}$ with
$$
S_i(x,y)=\left(e^{-\frac{1}{a}}x+c_i,\; e^{-\frac{1}{a+b}}y+d_{\pi(i)}\right),  \quad i\in \A,
$$
where $(c_i)_{i\in \A}$ and $(d_j)_{j\in \D}$ are chosen so that $S_i([0,1]^2)$'s are rectangles inside $[0,1]^2$ distributed as in  Figure \ref{Fig1}. Such sets belong to a broader class of self-affine sets studied by Lalley and Gatzouras in \cite{LaGa92}.

\begin{figure}
\begin{center}
\includegraphics[height=7cm]{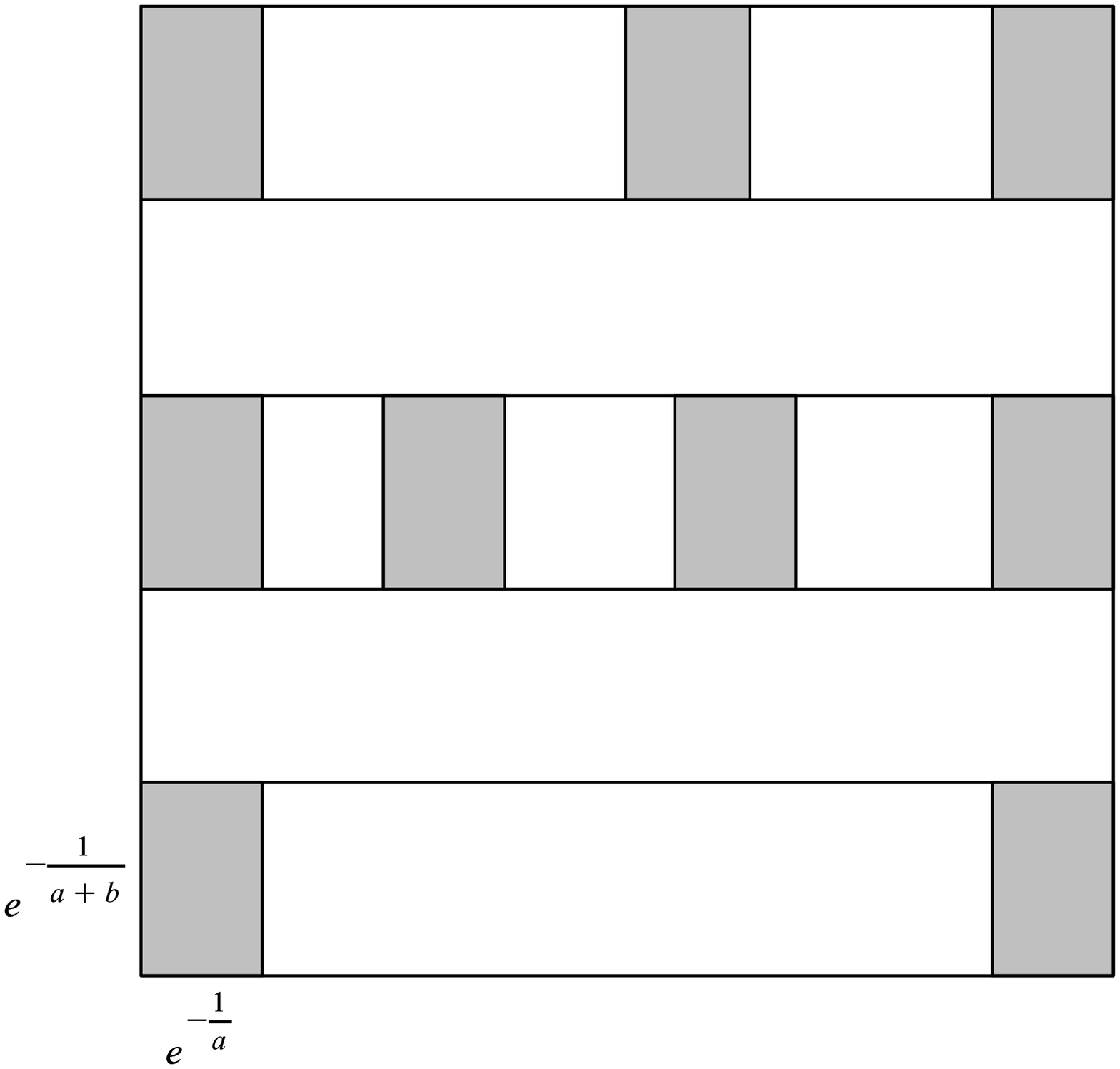}
\end{center}
\caption{}
\label{Fig1}
\end{figure}

For $\mu\in \M(X,T)$, define the  set of generic points of $\mu$ as
\begin{equation}\label{generic}
\mathcal{G}(\mu)=\left\{x\in X:
\lim_{n\to\infty} \frac{S_ng(x)}{n}=\int g d\mu, \ \forall\;
g\in C(X)\right\},
\end{equation}
where  $C(X)$ denotes the collection of real continuous functions on $X$, and $S_ng(x)=\sum_{i=0}^{n-1}g(T^ix)$.

At first, we deal with the Hausdorff dimension of the sets of generic points of invariant measures. When $\ba=(1,0)$, this result is well known (cf. \cite{Bow73,Caj81, PsiSul07,Fan08}).
\begin{thm}\label{thm-0.6}
Let $\mu\in \M(X,T)$. We have $\mathcal{G}(\mu)\neq \emptyset$ and
$\dim_H \mathcal{G}(\mu)=ah_\mu(T)+bh_{\mu\circ \pi^{-1}}(S)$.
\end{thm}

Next we consider the level sets  for Birkhoff averages of asymptotically
additive potentials on $X$. Let ${\bf
\Phi}=(\Phi_1,\ldots,\Phi_d)\in \C_{asa}(X,T)^d$, where $\Phi_i=(\log
\phi_{n,i})_{n=1}^\infty$. For $\alpha=(\alpha_1,\ldots,\alpha_d)\in
\R^d$, define
$$
E_{{\bf \Phi}}(\alpha)=\big\{x\in X: \lim_{n\to\infty} \frac 1n \log
\phi_{n,i}(x)=\alpha_i \mbox{ for } 1\leq i\leq d\big\}.
$$
For $\mu\in \M(X,T)$, write
$${\bf
\Phi}_*(\mu)=((\Phi_1)_*(\mu),\ldots, (\Phi_d)_*(\mu)).$$
 Define
$$L_{\bf \Phi}=\{{\bf \Phi}_*(\mu): \mu\in\mathcal{M}(X,T)\}.$$

\begin{thm}\label{thm-0.7} For $\alpha\in \R^d$,
$E_{\bf \Phi}(\alpha)\neq\emptyset$ if and only if $\alpha\in L_{\bf
\Phi}$. Furthermore for $\alpha\in L_{\bf \Phi}$, we have
\begin{equation*}
\begin{split}
\dim_HE_{\bf \Phi}(\alpha)&=\max\{ah_\mu(T)+bh_{\mu\circ
\pi^{-1}}(S):\; \mu\in \M(X,T),\; {\bf \Phi}_*(\mu)
=\alpha\}\\
&=\inf \{P^{\ba}(T,  \bq\cdot {\bf \Phi})-\alpha\cdot
\bq:\; \bq\in \R^d\},
\end{split}
\end{equation*}
where $\bq\cdot {\bf \Phi}$ denotes the potential $\sum_{i=1}^d
q_i\Phi_i$ for $\bq=(q_1,\ldots,q_d)$, and $\alpha\cdot \bq$ denotes the standard inner product of $\alpha$ and $\bq$. Moreover, if $L_\Phi$ is not reduced to a singleton, then $\{x\in X:\lim_{n\to\infty}{\bf \Phi}_n(x)/n\text{ does not exists}\}$ is of full Hausdorff dimension.
\end{thm}

The above theorem can be extended in an elaborated way. Let ${\bf \Phi}^{(1)}, {\bf \Phi}^{(2)}\in \C_{asa}(X,T)^d$, where ${\bf \Phi}^{(j)}=\big(\Phi_1^{(j)},\ldots, \Phi_d^{(j)}\big)$ with $\Phi_i^{(j)}=\big(\log \phi_{n,i}^{(j)}\big)_{n=1}^\infty\in \C_{asa}(X,T)$. Let $\bc=(c_1,c_2)\in \R^2$, where $c_1, c_2>0$. Denote for $\alpha
=(\alpha_1,\ldots,\alpha_d)\in \R^d$,
$$
E_{{\bf \Phi}^{(1)}, {\bf \Phi}^{(2)}, \bc}(\alpha)=\Big\{x\in X: \lim_{n\to\infty} \sum_{j=1}^2 \frac{1}{\lfloor c_jn\rfloor } \log
\phi^{(j)}_{\lfloor c_jn\rfloor ,i}(x)=\alpha_i \mbox{ for } 1\leq i\leq d \Big\},
$$
where $\lfloor c_jn\rfloor $ denotes the integral part of $c_jn$.

\begin{thm}\label{thm-0.8} Under the above setting, set ${\bf \Phi}=\sum_{j=1}^2{\bf \Phi}^{(j)}$. Then for $\alpha\in \R^d$,
$$E_{{\bf \Phi}^{(1)}, {\bf \Phi}^{(2)}, \bc}(\alpha)\neq\emptyset\Longleftrightarrow E_{\bf \Phi}(\alpha)\neq \emptyset\Longleftrightarrow\alpha\in L_{\bf
\Phi}.$$ Furthermore for $\alpha\in L_{\bf \Phi}$, we have
\begin{equation*}
\begin{split}
\dim_HE_{{\bf \Phi}^{(1)}, {\bf \Phi}^{(2)}, \bc}(\alpha)&=\dim_HE_{\bf \Phi}(\alpha)\\
&=\max\{ah_\mu(T)+bh_{\mu\circ
\pi^{-1}}(S):\; \mu\in \M(X,T),\; {\bf \Phi}_*(\mu)
=\alpha\}\\
&=\inf \{P^{\ba}(T,  \bq\cdot {\bf \Phi})-\alpha\cdot
\bq:\; \bq\in \R^d\}.
\end{split}
\end{equation*}
Moreover, if $L_\Phi$ is not reduced to a singleton, then $X\backslash \bigcup_{\alpha\in L_{{\bf \Phi}}}E_{{\bf \Phi}^{(1)}, {\bf \Phi}^{(2)}, \bc}(\alpha) $ is of full Hausdorff dimension.
\end{thm}

The level sets $E_{{\bf \Phi}^{(1)}, {\bf \Phi}^{(2)}, \bc}(\alpha)$ do depend on $\bc$ (see Example \ref{ex-5.7}). However, by Theorem \ref{thm-0.8},   $\dim_HE_{{\bf \Phi}^{(1)}, {\bf \Phi}^{(2)}, \bc}(\alpha)$ does not depends on $\bc$. It is quite interesting. As a natural application,
we shall use Theorem \ref{thm-0.8} to  study the multifractal analysis of
certain measures on $X$. Let $\Phi=(\log
\phi_n)_{n=1}^\infty\in\C_{asa}(X,T)$.
A probability measure $\mu$ is called an {\it $\ba$-weighted weak
Gibbs measure of $\Phi$} if there exists a sequence
$(\kappa_n)_{n=1}^\infty$ of positive numbers with $\lim_{n\to
\infty}(1/n)\log \kappa_n=0$, such that
$$
A(I)/\kappa_n\leq \mu(I)\leq \kappa_n A(I),  \quad I\in \A^n,
$$
where $A(I):=\exp\left( \frac{-nP^{\ba}(T,\Phi)}{a+b} \right)
\frac{\phi(I)^{1/a}} {\psi(\pi I)^{b/(a+b)}}$ is the term in the
right hand side of \eqref{e-2.2}. We recover the usual weak Gibbs
measures when $\ba=(1,0)$ and $\Phi$ is the sequence of Birkhoff
sums associated with a continuous potential over $X$ (cf. \cite{Yur97, Kes01}).  Our last
theorem is the following.

\begin{thm}
\label{thm-0.9}
Let $\Phi=(\log \phi_n)_{n=1}^\infty\in\C_{asa}(X,T)$. Then
 there exists at least one
$\ba$-weighted weak Gibbs measure of $\Phi$. Let $\mu$ be such a
measure. For $\alpha\geq 0$ we define
$$
E_\mu(\alpha)=\Big\{x\in X: \lim_{r\to 0^+}\frac{\log
\mu(B(x,r))}{\log r}=\alpha\Big \}.
$$
Let $\Psi_1=(\log\mu(x_{|n}))_{n=1}^\infty$, $\Psi_2=(\log
\mu\circ\pi^{-1}(\pi x_{|n}))_{n=1}^\infty$, and
$\Psi=a\Psi_1+b\Psi_2$, where $x_{|n}:=x_1\ldots x_n$ for $x=(x_i)_{i=1}^\infty\in X$. Then  $\Psi_1$, $\Psi_2$ and $\Psi$ belong
to $\C_{asa}(X,T)$.
Furthermore,  let $L_\mu=L_{-\Psi}=\{-\Psi_*(\lambda):
\lambda\in\mathcal{M}(X,T)\}$. Then, for all $\alpha\geq 0$,
$E_\mu(\alpha)\neq\emptyset$ if and only if $\alpha\in L_\mu$.
For $\alpha\in L_\mu$, we have
\begin{equation*}
\begin{split}
\dim_HE_\mu(\alpha)&=\sup\{ah_\lambda(T)+bh_{\lambda\circ \pi^{-1}}(S):\;
\lambda\in \M(X,T),\; \Psi_*(\lambda)=-\alpha\}\\
&=\inf \{P^{\ba}(T,  q\Psi)+\alpha q:\; q\in \R\}.
\end{split}
\end{equation*}
\end{thm}

\begin{rem}
{\rm  It is worth  mentioning that the {\it concatenation of measures} play a crucial role in our geometric results. At first, the computations of Hausdorff dimensions are based on a kind of constructions of Moran measures obtained by the concatenation
of quasi-Bernoulli measures. This method strongly depends on Theorem~\ref{thm-0.5}. In the classical case for which $b=0$, one can construct either Moran measures  by concatenating Markov measures (see e.g. \cite{Caj81}), or Moran sets directly (see for instance \cite{FFW01,Feng03}). This second approach  seems not efficient when $b\neq 0$.

Also, the existence of (weighted) weak Gibbs measures
for a given asymptotically additive potential $\Phi$ is obtained by concatenating  (weighted) Gibbs measures associated with H\"older potentials converging to $\Phi$.}
\end{rem}

\begin{rem}

{\rm

\begin{itemize}

\item[(1)] We mention that \eqref{e-0.5} is obtained independently in \cite{Yay09} for $\Phi=0$.

\item[(2)] Theorem \ref{thm-0.2} has been partially extended  in \cite{Fen09} to the case that $X$ is a subshift satisfying specification. For example, the uniqueness of weighted equilibrium states is proved for  almost additive potentials with the bounded distortion condition. This solves a question of
    Gaztouras and Peres about the uniqueness of invariant measures of maximizing weighted entropy (cf.  \cite[Problem 3]{GaPe96}).

\item[(3)]

Special cases of Theorems~\ref{thm-0.7}
and~\ref{thm-0.9}  have been obtained in
\cite{BaMe08} and \cite{Kin95,BaMe08} respectively when $d=1$ and
under the bounded distortion assumption, except for the endpoints of
the spectra which are not captured by the methods developed in these
papers. Moreover, those methods cannot be extended to the case of
general almost additive potentials. Also, the results on
multifractal analysis of Birkhoff averages and quasi-Bernoulli
measures in those papers are not unified, while it is the case in the self-similar
case $b=0$. The weighted thermodynamic formalism introduced in this
paper makes it possible to have a simple and unified presentation of
the results concerning both questions.

\item[(4)] Reduced to the case $b=0$, Theorems \ref{thm-0.7}-\ref{thm-0.8} cover the previous works
on the multifractal analysis of almost
additive potentials and related measures on symbolic spaces with   the  standard  metric (see \cite{Rand89,Pes97, BaSc00, FFW01,Feng03,Bar09} and references therein).

\item[(5)] Following the works achieved in \cite{Kin95,Ols98,BaMe07} for almost additive potentials satisfying the bounded distortion property, it is possible to conduct
the multifractal analysis of the projections of weak Gibbs measures on the planar self-affine sets described above when conditions \eqref{planarset} hold. We will not discuss such geometrical realizations in this paper.


\item[(6)] It is worth to point out that  Falconer gave a variational formula for the Hausdorff dimension for
 ``almost all'' self-affine sets under some assumptions \cite{Fal88}, and for this case K\"{a}enm\"{a}ki showed the existence of  ergodic measures of full Hausdorff dimension on the typical self-affine sets \cite{Kae04}. See \cite{JoSi07} for a related result on the multifractal analysis.

\end{itemize}
}
\end{rem}

The paper is organized  as follows. Some definitions and  known
results on sub-additive thermodynamic formalism on subshifts are
given in Section~\ref{S2}. The proofs of Theorems~\ref{thm-0.1}--\ref{thm-0.5} on the weighted
thermodynamic formalism are given in Section~\ref{Proofsth1.1to1.5}. In Section~\ref{S7}, we present the higher
dimensional weighted thermodynamic formalism. Since the proofs of the result are very similar to those used in the 2-dimensional case, we omit them. Then, in Section~\ref{multifractalanalysis} we present and prove the extensions to the higher dimensional case of Theorems~\ref{thm-0.6}--\ref{thm-0.9}. Indeed, for these results, the higher dimensional case is more involved, due to the upper bound estimates for Hausdorff dimensions.

\section{Sub-additive thermodynamical formalism on subshifts}
\label{S2} In this section, we present some definitions and known
results about sub-additive thermodynamical formalism on subshifts.
\subsection{One-sided subshifts over finite alphabets}
\label{S-2.1} Let $p\geq 2$ be an integer and
$\A=\{1,\ldots,p\}$. Denote
$$
\A^\N=\left\{(x_i)_{i=1}^\infty:\; x_i\in \A\mbox { for }i\geq 1\right\}.
$$
Then $\A^\N$ is compact endowed with the product discrete topology (\cite{LiMa95}).
We say that $(X,T)$ is a {\it subshift over $\A$}, if $X$ is a
compact subset of $\A^\N$ and $T(X)\subseteq X$, where $T$ is the
left shift map on $\A^\N$  defined as
$$
T((x_i)_{i=1}^\infty)=(x_{i+1})_{i=1}^\infty,\quad\forall\; (x_i)_{i=1}^\infty\in \A^\N.
$$
In particular, $(X,T)$ is called the {\it full shift over $\A$} if
$X=\A^\N$.  For any $n\in \N$ and $I\in \A^n$, we write
$$
[I]=\{(x_i)_{i=1}^\infty\in \A^\N:\; x_1\ldots x_n=I\}
$$
and call it an {\it $n$-th cylinder} in $\A^\N$.

Let $(X,T)$ and  $(Y,S)$ be two subshifts over finite alphabets $\A$
and  $\D$, respectively. We say that $Y$ is a {\it factor} of $X$,
if there is a continuous surjective map $\pi:\; X\to Y$ such that
$\pi T=S\pi$.  Here $\pi$ is called a {\it factor map}. Furthermore
$\pi$ is called a {\it one-block factor map} if there exists a map
${\pi}:\; \A\to \D$  such that
$$
\pi\left((x_i)_{i=1}^\infty\right)=\left({\pi}(x_i)\right)_{i=1}^\infty,\qquad
\forall \;(x_i)_{i=1}^\infty\in X.
$$
It is well known (see, e.g. \cite[Proposition 1.5.12]{LiMa95}) that  each factor map
$\pi:\; X\to Y$ between two subshifts $X$ and $Y$, will become a
one-block factor map if we enlarge the alphabet for $X$ and  recode $X$
appropriately.

\subsection{Sub-additive thermodynamical formalism}
For the reader's convenience we recall some definitions. Let $(X,T)$ be a subshift over a finite alphabet $\A$. A sequence
$\Phi=(\log \phi_n)_{n=1}^\infty$  is called a {\it sub-additive
potential} on $X$ and write $\Phi\in \C_s(X,T)$, if  each $\phi_n$ is a non-negative continuous
function on $X$ and there exists $c>0$ such that
\begin{equation*}
\label{e-1.1}
\phi_{n+m}(x)\leq c \phi_n(x)\phi_m(T^nx),\quad \forall \; x\in X,\; n,m\in \N.
\end{equation*}
More generally,   $\Phi=(\log \phi_n)_{n=1}^\infty$ is said to be an {\it asymptotically sub-additive potential} and  write $\Phi\in \C_{ass}(X,T)$ if for any $\varepsilon>0$, there exists a sub-additive potential $\Psi=(\log \psi_n)_{n=1}^\infty$ on $X$ such that
$$
\limsup_{n\to \infty}\frac{1}{n}\sup_{x\in X} |\log \phi_n(x)-\log \psi_n(x)|\leq \varepsilon,
$$
where we take the convention $\log 0-\log 0=0$.
Furthermore $\Phi$ is called an {\it asymptotically additive potential} and  write $\Phi\in \C_{asa}(X,T)$
if both $\Phi$ and $-\Phi$ are asymptotically sub-additive, where $-\Phi$ denotes $(\log (1/\phi_n))_{n=1}^\infty$.

Let $\M(X,T)$ denote the set of $T$-invariant Borel
probability measures on $X$ endowed with the weak-star topology. For
$\mu\in \M(X,T)$, let $h_\mu(T)$ denote the measure-theoretic
entropy of $\mu$ with respect to $T$, and  write
\begin{equation}
\label{e-2.1T2}
 \Phi_*(\mu)=\lim_{n\to \infty}\frac{1}{n}\int_X \log
\phi_n(x)\; d\mu(x).
\end{equation}
The existence of the limit (which may take value $-\infty$) in (\ref{e-2.1T2}) follows from the
sub-additivity of $\Phi$. The following lemma will be useful.

\begin{lem}[\cite{FeHu09a}]
\label{lem-2.1}
Let $\Phi=(\log \phi_n)_{n=1}^\infty\in \C_{ass}(X,T)$.
Then we have the following properties.
\begin{itemize}
\item[(i)] Let $\mu\in \M(X,T)$.   The limit $\lambda_\Phi(x):=\lim_{n\to \infty}\frac{1}{n}\log \phi_n(x)$
exists (which may take value $-\infty$) for $\mu$-a.e.~$x\in X$, and $\int \lambda_\Phi(x)\;d\mu(x)=\Phi_*(\mu)$. When $\mu$ is ergodic,
$\lambda_\Phi(x)=\Phi_*(\mu)$ for $\mu$-a.e. $x\in
X$.

\item[(ii)] The map $\Phi_*:\mathcal{M}(X,T)\rightarrow \mathbb{R}\cup
\{-\infty\}$ is upper semi-continuous,  and there is $C\in
\mathbb{R}$ such that  for all $\mu\in
\mathcal{M}(X,T)$, $\lambda_\Phi(x)\le C$ $\mu$-a.e and $\Phi_*(\mu)\le C$.
 If  $\Phi\in \C_{asa}(X,T)$,  $\Phi_*$ is continuous on $\mathcal{M}(X,T)$.

\item[(iii)] $\Phi\in \C_{asa}(X,T)$ if and only if for any $\varepsilon >0$, there exists a continuous function $g$ on $X$ such that
\begin{equation*}
\label{e-1504f}
\limsup_{n\to \infty}\frac{1}{n}\sup_{x\in X}|\log \phi_n(x)-S_ng(x)|\leq \varepsilon,
\end{equation*}
where $S_ng(x):=\sum_{j=0}^{n-1}g(T^jx)$.
\end{itemize}
\end{lem}

\begin{rem}
\label{re-2.2}
{\rm
According to Lemma \ref{lem-2.1}(iii), for $\mu\in \M(X,T)$, the set $\G(\mu)$ of generic points of $\mu$ defined as in \eqref{generic} is just equal to
$$\left\{x\in X:\; \lim_{n\to \infty} \frac{\log \phi_n(x)}{n}=\Phi_*(\mu),\quad \forall \; \Phi=(\log \phi_n)_{n=1}^\infty \in \C_{asa}(X,T)\right\}.$$
}
\end{rem}

For $\Phi=(\log \phi_n)_{n=1}^\infty\in \C_{ass}(X,T)$, and a compact
set $K\subseteq X$,  define
\begin{equation} \label{e-2.1T1}
P_n(T,\Phi, K)=\sum_{I\in \A^n,\; [I]\cap K\neq \emptyset}
\sup_{x\in [I]\cap K}\phi_n(x). \end{equation}
 and
\begin{equation}\label{e-2.1T4}
P(T, \Phi,K)=\limsup_{n\to \infty}\frac{1}{n}\log P_n(T,\Phi, K).
\end{equation}

The following variational principle was
proved in \cite{CFH08} when $\Phi\in \C_s(X,T)$. As pointed in \cite{FeHu09a}, it holds also for
$\Phi\in \C_{ass}(X,T)$.

\begin{pro}
\label{pro-1.1} Let $P(T,\Phi,X)$ be defined as above. Then for any
$\Phi \in \C_{ass}(X,T)$, we have the following variational
principle:
\begin{equation}
\label{e-1.2}
P(T,\Phi,X)=\sup\{\Phi_*(\mu)+h_\mu(T):\; \mu\in \M(X,T)\}.
\end{equation}
\end{pro}
We call   $P(T,\Phi):=P(T,\Phi,X)$ the {\it topological pressure of
$\Phi$}.
\begin{rem}
{\rm When $\Phi=(\log \phi_n)_{n=1}^\infty$ is an additive potential,
i.e., $$\phi_n(x)=\exp\left(\sum_{i=0}^{n-1}\phi(T^ix)\right)$$ for
a continuous function $\phi$ on $X$,
 the above proposition comes to  the Ruelle-Walters variational principle for additive
 topological pressures (see e.g. \cite{Rue73, Rue78, Wal75}).}
\end{rem}

We say that $\mu\in \M(X,T)$ is an {\it equilibrium state} of $\Phi$
if the supremum in (\ref{e-1.2}) is attained at $\mu$.  Note that
$\Phi_*(\cdot)$ is  upper semi-continuous on $\M(X,T)$ (cf. Lemma~\ref{lem-2.1}(ii)), and so is $h_{(\cdot)}(T)$ for subshifts.
Hence $\Phi$ has at least one equilibrium state. In the following,
we consider the case when $\Phi$ has a unique equilibrium state.

We say that
$\Phi=(\log \phi_n)_{n=1}^\infty$ is {\it almost additive} if
$\phi_n$ is positive and continuous on $X$  for each $n$ and  there
is a constant $c>0$ such that
\begin{equation*}
\label{e-taa}
\frac{1}{c} \phi_n(x)\phi_m(T^nx)\leq \phi_{n+m}(x)\leq c
\phi_n(x)\phi_m(T^nx),\quad \forall \; x\in X,\; n,m\in \N.
\end{equation*}
For convenience, we denote by  $\C_{aa}(X,T)$ the collection of
almost-additive potentials on $X$.  Clearly $\C_{aa}(X,T)\subset \C_{asa}(X,T)$.

For $\Phi=(\log \phi_n)_{n=1}^\infty\in \C_{ass}(X,T)$, we say that $\Phi$ has the {\it bounded distortion property} if there exists a constant $c>0$ such that
$$
\frac{1}{c}\phi_n(y)\leq \phi_n(x)\leq c\phi_n(y)\quad \mbox{whenever $x,y\in X$ are in the same $n$-th cylinder}.
$$

\begin{pro}
\label{pro-1.2} Let $(X,T)$ be a full shift or mixing subshift of
finite type. Let $\Phi=(\log \phi_n)_{n=1}^\infty \in \C_{aa}(X,T)$.
Assume that $\Phi$ has the bounded distortion property. Then $\Phi$
has a unique equilibrium state $\mu$. Furthermore, there exists a
constant $c>0$ such that for any $n\in \N$ and
$x=(x_i)_{i=1}^\infty\in X$,
 \begin{equation*}
\label{e-1.3}
c^{-1}\leq \frac{\mu([x_1\ldots x_n])}{\exp(-n P(T,\Phi))\;\phi_n(x)}\leq c.
\end{equation*}
\end{pro}
 Proposition \ref{pro-1.2}  was first proved in
\cite{FeLa02, Fen04} for  special almost additive potentials given by
$$\phi_n(x)=\|M(x)M(Tx)\ldots M(T^{n-1}x)\|,\quad n\in \N,$$
  where $M$ is
a H\"{o}lder continuous function taking values in the set of
$d\times d$ positive matrices. It was completed into the present
form by Barreira \cite{Bar06} and Mummert \cite{Mum06}
independently. We remark that Proposition \ref{pro-1.2} extends the
classical theory about equilibrium states for additive continuous
 potentials with the bounded distortion property (cf. Bowen \cite{Bow75}).

\subsection{Relativized sub-additive thermodynamic formalism}
Let $\pi:\; X\to Y$ be a one-block factor map between two subshifts
$(X,T)$ and $(Y,S)$.  The following relativized variational
principle  was proved in \cite{CaZh08} for sub-additive potentials
under a general random setting by using an idea in
\cite{CFH08}. It does hold for $\Phi\in \C_{ass}(X,T)$ by modifying the proof in \cite{CaZh08} slightly. This extends the relativized variational principle  of Ledrappier and Walters
\cite{LeWa77} for additive potentials.

\begin{pro}\label{pro-1.3}
Let $\Phi\in \C_{ass}(X,T)$ and $\nu\in \M(Y,S)$. Then
\begin{equation}
\label{e-1.4}
\sup\{\Phi_*(\mu)+h_\mu(T)-h_\nu(S)\}=\int_Y P(T,\Phi,\pi^{-1}(y))\; d\nu(y),
\end{equation}
where the supremum is taken over the set of $\mu\in \M(X,T)$ such
that $\mu\circ \pi^{-1}=\nu$, $P(T,\Phi,\pi^{-1}(y))$ is defined as
in \eqref{e-2.1T4}.
\end{pro}

By the upper semi-continuity of $\Phi_*(\cdot)$ and $h_{(\cdot)}(T)$
on $\M(X,T)$, the supremum in (\ref{e-1.4}) is attainable. Any
measure $\mu\in \M(X,T)$ for which the supremum in  (\ref{e-1.4}) is
attained at $\mu$ is called a {\it conditional equilibrium state of
$\Phi$ with respect to $\nu$}.

\section{Weighted thermodynamic formalism}\label{Proofsth1.1to1.5}
\subsection{The proof of Theorem \ref{thm-0.1}}\label{S3}

 Throughout this
section, we assume that $X$ is a subshift over $\A$, $Y$ a subshift
over $\D$ and $\pi:\; X\to Y$  a one-block factor map. The following
lemma plays a key role in the proof of Theorem \ref{thm-0.1}.
\begin{lem}
\label{lem-3.1} Let $\Phi=(\log \phi_n(x))_{n=1}^\infty\in
\C_{ass}(X,T)$ and $\nu\in \M(Y, S)$. Then we have
\begin{equation}
\label{e-T5} \sup\{\Phi_*(\mu)+h_\mu(T)-h_\nu(S):\;\mu\in \M(X,T),
\;\mu\circ \pi^{-1}=\nu\}=\Psi_*(\nu),
\end{equation}
where $\Psi=(\log \psi_n)_{n=1}^\infty\in \C_{ass}(Y,S)$ is defined
by
\begin{equation*}
\label{e-1.4''} \psi_n(y)=\sum_{I\in \A^n:\; [I]\cap \pi^{-1}(y)\neq
\emptyset} \sup_{x\in [I]\cap\pi^{-1}(y)}\phi_n(x).
\end{equation*}
\end{lem}
\begin{proof} By Proposition \ref{pro-1.3}, the left-hand
side of \eqref{e-T5} equals $\int P(T, \Phi, \pi^{-1}(y))\;
d\nu(y)$. However by \eqref{e-2.1T4}-\eqref{e-2.1T1},
$$
P(T,\Phi,\pi^{-1}(y))=\limsup_{n\to \infty}\frac{1}{n}\log
P_n(T,\Phi,\pi^{-1}(y))$$ and $$P_n(T,\Phi,\pi^{-1}(y))=\sum_{I\in
\A^n:\; [I]\cap \pi^{-1}(y)}\sup_{x\in [I]\cap
\pi^{-1}(y)}\phi_n(x).
$$
Clearly $\psi_n(y)=P_n(T,\Phi,\pi^{-1}(y))$. It is direct to check
that $\Psi=(\log \psi_n)_{n=1}^\infty\in \C_{ass}(Y,S)$. Hence by
Lemma \ref{lem-2.1},
$$\Psi_*(\nu)=\int \limsup_{n\to \infty}\frac{1}{n}\log \psi_n(y)\; d\nu(y)
=\int P(T,\Phi,\pi^{-1}(y))\; d\nu(y).$$ This finishes the proof of
the lemma.
\end{proof}

\begin{proof}[Proof of Theorem \ref{thm-0.1}] Clearly
we have
\begin{equation}
\label{e-3.1T}
\begin{split}
\sup & \{\Phi_*(\mu)+a h_\mu(T)+b
h_{\mu\circ \pi^{-1}}(S):\;\mu\in \M(X,T)\}\\
&=\sup \{\Phi_*(\mu)+a h_\mu(T)+b h_{\nu}(S):\;\nu\in \M(Y,S),
\;\mu\in \M(X,T),\;\mu\circ \pi^{-1}=\nu\}\\
&=\sup\{A(\nu)+(a+b)h_\nu(S): \; \nu\in \M(Y,S)\},
\end{split}
\end{equation}
where
$A(\nu):=a\sup\{\frac{1}{a}\Phi_*(\mu)+h_\mu(T)-h_{\nu}(S):\;\mu\in
\M(X,T),\; \mu\circ \pi^{-1}=\nu\}$.

By Lemma \ref{lem-3.1}, we have $A(\nu)=a\Psi_*(\nu)$, where
$\Psi=(\log \psi_n)_{n=1}^\infty\in \C_{ass}(Y,S)$ is defined as
\begin{equation*}
\label{e-1.4''} \psi_n(y)=\sum_{I\in \A^n:\; [I]\cap \pi^{-1}(y)\neq
\emptyset} \sup_{x\in [I]\cap\pi^{-1}(y)}\phi_n(x)^{1/a}.
\end{equation*}
Hence by \eqref{e-3.1T} and Proposition \ref{pro-1.1}, we have
\begin{equation}
\label{e-3.2T}
\begin{split}
\sup & \{\Phi_*(\mu)+a h_\mu(T)+b
h_{\mu\circ \pi^{-1}}(S):\;\mu\in \M(X,T)\}\\
&=\sup\{a\Psi_*(\nu)+(a+b)h_\nu(S): \; \nu\in \M(Y,S)\}\\
&=(a+b)\sup\left\{\frac{a}{a+b}\Psi_*(\nu)+h_\nu(S): \; \nu\in
\M(Y,S)\right\}\\
&=(a+b)P\left(S, \frac{a}{a+b}\Psi\right)=\lim_{n\to
\infty}\frac{a+b}{n}\log \sum_{J\in \D^n}\sup_{y\in [J]\cap
Y}\psi_n(y)^{\frac{a}{a+b}}.
\end{split}
\end{equation}
This proves the first part of Theorem \ref{thm-0.1}. The second part
follows directly from \eqref{e-3.2T} and \eqref{e-3.1T}.
\end{proof}

\subsection{The proof of Theorem \ref{thm-0.2}}\label{S4}

Throughout this section, we assume that  $X=\A^\N$ and $Y=\D^\N$ are
two full shifts over finite alphabets, and $\pi:\; X\to Y$ is a
one-block factor map. To prove Theorem \ref{thm-0.2}, we need some
auxiliary results.

\begin{lem}
\label{lem-1.4}
Assume that
 $\Phi\in \C_{aa}(X,T)$ and that $\Phi$ satisfies the bounded distortion property. Let $\nu\in \M(Y,S)$.
 Then $\int_Y P(T,\Phi,\pi^{-1}(y))\; d\nu(y)=\Psi_*(\nu)$, where $\Psi=(\log \psi_n)_{n=1}^\infty\in \C_{aa}(Y,S)$ is given by
 \begin{equation*}
 \label{e-4.1T}
 \psi_n(y)=\sum_{I\in \A^n:\; \pi I=y_1\ldots y_n} \sup_{x\in [I]}\phi_n(x),\quad \forall \; y=(y_i)_{i=1}^\infty\in Y.
 \end{equation*}
 Furthermore
 \begin{equation*}
\label{e-1.4'}
\sup\{\Phi_*(\mu)+h_\mu(T)-h_\nu(S)\}=\Psi_*(\nu),
\end{equation*}
where the supremum is taken over the set of $\mu\in \M(X,T)$ such that $\mu\circ \pi^{-1}=\nu$.
\end{lem}
\begin{proof}
It follows directly from Lemma \ref{lem-3.1} and the bounded
distortion property of $\Phi$.
\end{proof}

\begin{pro}\label{pro-1.4} Assume that
 $\Phi\in \C_{aa}(X,T)$ and $\Phi$ satisfies the bounded distortion property. Let $\nu\in \M(Y,S)$ so that $\nu$ has the quasi-Bernoulli property. Then $\Phi$ has a unique conditional equilibrium state $\mu$ with respect to $\nu$. Furthermore there is a constant $c>0$ such that
\begin{equation}
\label{e-1.5} c^{-1}\leq \frac{\mu(I)}{ \nu(\pi I)
\phi(I)/\psi(\pi I)}\leq c, \quad \forall\; n\in \N,\; I\in \A^n, \;
J\in \D^n,
\end{equation}
where $\phi(I):=\sup_{x\in [I]}\phi_n(x)$ for $I\in \A^n$ and
$\psi(J):=\sum_{I\in \A^n:\; \pi I=J}\phi(I)$ for $J\in \D^n$.
\end{pro}

\begin{proof}
We first construct  $\mu\in \M(X,T)$ such that $\mu\circ \pi^{-1}=\nu$ and  $\mu$ satisfies (\ref{e-1.5}). Here we adopt  an idea from  \cite{FeLa02}.
Since   $\Phi\in \C_{aa}(X,T)$ and $\Phi$ satisfies the bounded distortion property, it is direct to check that
$\phi$ and $\psi$ are quasi-Bernoulli in the sense that
$$
\phi(I_1I_2)\approx \phi(I_1)\phi(I_2),\quad I_1, I_2\in
\A^*=\bigcup_{n\geq 1}\A^n,
$$
and
$$
\psi(J_1J_2)\approx \psi(J_1)\psi(J_2),\quad J_1, J_2\in
\D^*=\bigcup_{n\geq 1}\D^n,
$$
where for two families of positive numbers $(a_n)$ and $(b_n)$,  we write $(a_n)\approx (b_n)$ if $a_n/b_n$ is bounded from below and above by some positive constants.

For each integer $n>0$, let ${\mathcal B}_n$ be the $\sigma$-algebra generated by the cylinders
 $[I]$ in $X$, $I\in \A^n$. We define a sequence of probability measures $(\mu_n)_{n=1}^\infty$ on  ${\mathcal B}_n$ by
\begin{equation*}
\label{e-1.6} \mu_n(I)=\nu(\pi I) \phi(I)/\psi(\pi I), \qquad
\forall \ I\in \A^n.
\end{equation*}
 Then there is a subsequence $(\mu_{n_k})_{k\geq 1}$ converging in the weak-star
 topology to a probability measure $\widetilde{\mu}$. We claim that
 $\widetilde{\mu}$ satisfies (\ref{e-1.5}). To see this,  for any $I\in \A^n$  and $p > n$, we have
\begin{eqnarray*}
 \mu_{p}(I) &=& \sum_{I_1\in \A^{p-n}}\mu_{p}(II_1)=
 \sum_{I_1\in \A^{p-n}}\nu(\pi (II_1))\phi(II_1)/\psi(\pi(II_1))\\
 &\approx &
 \frac{\nu([\pi I])\phi(I)}{\psi(\pi I)}  \sum_{I_1\in \A^{p-n}} \frac{\nu(\pi I_1)\phi(I_1)}{\psi(\pi I_1)}
 =  \nu(\pi I)\phi(I)/\psi(\pi I).
 \end{eqnarray*}
Letting $p  = n_k \uparrow  \infty$, we obtain
$\widetilde{\mu}(I)\approx \nu(\pi I)\phi(I)/\psi(\pi I)$, as
desired.

 Let  $\mu$ be  a limit point
 of the sequence  $\frac{1}{n}\left(\widetilde{\mu}+\widetilde{\mu}\circ T^{-1}
 +\ldots +\widetilde{\mu}\circ T^{-(n-1)}\right)$ in the weak-star
 topology. Then $\mu\in \M(X,T)$ (cf. \cite[Theorem 6.9]{Wal82}).  Note that for any $I\in \A^n$   and $p\geq 0$,
\begin{eqnarray*}
  \widetilde{\mu}\circ T^{-p}(I) & = & \sum _{I_1\in \A^p}\widetilde{\mu}(I_1I) \approx
 \sum_{I_1\in \A^{p}}\nu(\pi(I_1I))\phi(I_1I)/\psi(\pi(I_1I))\\
 &\approx &
 \frac{\nu(\pi I)\phi(I)}{\psi(\pi I)} \sum_{I_1\in \A^{p}}\frac{\nu(\pi I_1)\phi(I_1)}{\psi(\pi I_1)}
 =  \nu(\pi I)\phi(I)/\psi(\pi I).
 \end{eqnarray*}
Hence we have $\mu(I) \approx \nu(\pi I)\phi(I)/\psi(\pi I)$. It
is clear that $\mu$ is quasi-Bernoulli.  Hence $\mu$ is ergodic
(cf. \cite[Theorem 1.5(iv)]{Wal82}). Also, by construction, we have $\mu\circ\pi^{-1}(\pi I)\approx \nu(\pi I)$ ($I\in \bigcup_{n\ge 1}\A^n$). Since both $\mu\circ\pi^{-1}$ and $\nu$ are ergodic, we have $\mu\circ\pi^{-1}=\nu$.

Next we show that $\mu$ is a conditional equilibrium state of $\Phi$ with respect to $\nu$.
Write for $n\in \N$,
$$
t_n=-\left(\sum_{I\in \A^n} \mu(I)\log
\mu(I)\right)+\left(\sum_{J\in \D^n}\nu(J)\log \nu(J)\right).
$$
Then $(t_n)_{n\ge 1}$ is sub-additive in the sense that $t_{n+m}\leq t_n+t_m$ for any $n, m\in \N$ (cf. \cite[Lemma 1]{DoSe02}),
and hence
\begin{equation}\label{hmu}
h_\mu(T)-h_\nu(S)=\lim_{n\to \infty} t_n/n=\inf_{n\in \N}t_n/n.
\end{equation}

For two families of real numbers $\{a_i\}_{i\in \I}$ and  $\{b_i\}_{i\in \I}$, we write  $a_i=b_i+O(1)$ if there is a constant $c>0$ such that
$|a_i-b_i|\leq c$ for each $i\in \I$.
By the quasi-Bernoulli property of $\mu$ and $\nu$, we have
\begin{equation*}
\begin{split}
&\int \log \phi_n(x)\; d\mu(x)+t_n\\
&\mbox{}\quad =O(1)+ \left(\sum_{I\in \A^n} \mu(I)\log \phi(I)-\mu(I) \log \mu(I)\right)\\
&\mbox{}\qquad\qquad\qquad\qquad \qquad\qquad+\sum_{J\in \D^n}\nu(J)\log \nu(J)\\
&\mbox{}\quad =O(1)+\sum_{J\in \D^n}\nu(J) \sum_{I\in \A^n:\; \pi
I=J}
\frac{\mu(I)}{\nu(J)}\log \frac{\phi(I)\nu(J)}{\mu(I)}\\
&\mbox{}\quad =O(1)+\sum_{J\in \D^n} \nu(J) \sum_{I\in \A^n:\; \pi
I=J}
\frac{\mu(I)}{\nu(J)} \log \psi(J) \quad (\mbox{by \eqref{e-1.5}})\\
&\mbox{}\quad =O(1)+\sum_{J\in \D^n} \nu(J)  \log \psi(J),\\
\end{split}
\end{equation*}
Dividing both sides by $n$ and letting $n\to \infty$,
we obtain
$$
\Phi_*(\mu)+h_\mu(T)-h_\nu(S)=\Psi_*(\nu).
$$
Hence by Lemma \ref{lem-1.4},  $\mu$ is a conditional equilibrium
state of $\Phi$ with respect to $\nu$.

In the end, we prove that  $\mu$ is the unique  conditional equilibrium state of $\Phi$ with respect to $\nu$.
Here we adopt an idea due to Bowen (cf. \cite[p. 34--36]{Bow75}). Assume that $\mu'\neq \mu$ is another   conditional equilibrium state of $\Phi$ with respect to $\nu$.
That is, $\mu'\circ \pi^{-1}=\nu$ and
\begin{equation}
\label{e-another}
\Phi_*(\mu')+h_{\mu'}(T)-h_\nu(S)=\Psi_*(\nu).
\end{equation}
Without loss of generality we may assume that $\mu'$ is ergodic
(otherwise, we may consider the  ergodic decomposition of $\mu'$). Then $\mu'$ and $\mu$ are totally singular to each other. Hence for each $\varepsilon>0$ and  sufficiently large  $n$, there exists a  set $F_n$ which is the union of some $n$-th cylinders in $X$, such that
\begin{equation}
\label{e-1.7}
\mu(F_n)<\varepsilon \quad \mbox{ and }\quad  \mu'(F_n)>1-\varepsilon.
\end{equation}
 It is direct to check that
$$
\Big|n\Psi_*(\nu)-\sum_{J\in \D^n} \nu(J)  \log \psi(J)\Big|=O(1)
\quad \mbox{and}
$$
$$
\mbox{for }\lambda\in\{\mu,\mu'\},\ \Big|n\Phi_*(\lambda)-\sum_{I\in
\A^n} \lambda(I)  \log \phi(I)\Big|=O(1).
$$
Hence for $\lambda\in\{\mu,\mu'\}$ we have
\begin{equation*}
\begin{split}
n\Phi_*(\lambda)&= \sum_{I\in \A^n} \lambda(I)  \log \phi(I)+O(1)\\
&= \sum_{I\in \A^n} \lambda(I)  \log \frac{\mu(I)\psi(\pi I)}{\nu(\pi I)}+O(1)\\
&= \left(\sum_{I\in \A^n} \lambda(I) \log \mu(I)\right)+
\left(\sum_{J\in \D^n}\nu(J)\log \frac{\psi(J)}{\nu(J)}\right)+O(1)\\
&= \left(\sum_{I\in \A^n} \lambda(I) \log \mu(I)\right)-
\left(\sum_{J\in \D^n}\nu(J)\log \nu(J)\right)+n\Psi_*(\nu)+O(1).\\
\end{split}
\end{equation*}
Hence, by \eqref{e-another} and applying \eqref{hmu} to $\mu'$  we have
\begin{equation*}
\label{}
\begin{split}
0&\leq n\Phi_*(\mu')-\left(\sum_{I\in \A^n} \mu'(I)\log  \mu'(I)\right)+
\left(\sum_{J\in \D^n}\nu(J)\log \nu(J)\right)-n\Psi_*(\nu)\\
&= \sum_{I\in \A^n} \Big[-\mu'(I)\log \mu'(I)+\mu'(I)\log
\mu(I)\Big]
+O(1)\\
&= \sum_{[I]\subset F_n} \Big[-\mu'(I)\log \mu'(I)+\mu'(I)\log
\mu(I)\Big]
\\
&\qquad+ \sum_{[I]\subset X\backslash F_n} \Big[-\mu'(I)\log  \mu'(I)+\mu'(I)\log  \mu(I)\Big]+O(1)\\
&\leq  \mu'(F_n)\log  \mu(F_n)+\mu'(X\backslash F_n)\log  \mu(X\backslash F_n)+2\sup_{0\leq s\leq 1}(-s\log s)+O(1),\\
\end{split}
\end{equation*}
where for the last inequality, we use the elementary inequality (cf.
\cite[Lemma 1.24]{Bow75})
$$
\sum_{i=1}^k (-p_i\log p_i+p_i\log a_i)\leq s \log \sum_{i=1}^k a_i -s\log s, \quad   s:=\sum_{i=1}^kp_i, \; p_i\geq 0.
$$
 It leads to a contradiction since by (\ref{e-1.7}),  $\mu'(F_n)\log  \mu(F_n)\to -\infty$ as $\varepsilon\to 0$.
This finishes the proof of Proposition \ref{pro-1.4}.
\end{proof}

\begin{proof}[Proof of Theorem \ref{thm-0.2}]
Assume that $\Phi=(\log \phi_n)_{n=1}^\infty\in \C_{aa}(X,T)$  satisfies the
bounded distortion property. Let $\ba=(a,b)\in \R^2$ so that $a>0$
and $b\geq 0$.

Write $\phi(I)=\sup_{x\in [I]} \phi_n(x)$ for $I\in \A^n$ and
$\psi(J)=\sup_{I\in \A^n:\; \pi I=J}\phi(I)^{1/a}$ for $J\in \D^n$.
Define $\Psi=(\log \psi_n)_{n=1}^\infty$ by
$\psi_n(y)=\psi(y_1\ldots y_n)$.  By the assumption on $\Phi$,  we
have $\Psi\in \C_{aa}(Y,S)$. By Theorem \ref{thm-0.1}
we have
\begin{equation*}
 P^\ba(T,\Phi)=(a+b)P\Big(S,\frac{a}{a+b}\Psi\Big)=\lim_{n\to
\infty}\frac{a+b}{n} \log \sum_{J\in \D^n}\Big(\sum_{I\in \A^n:\;
\pi I=J}\phi(I)^{\frac 1a} \Big)^{\frac{a}{a+b}}.
\end{equation*}

Let $\mu$ be  an $\ba$-weighted equilibrium state of $\Phi$ and
$\nu=\mu\circ \pi^{-1}$. By Theorem \ref{thm-0.1}, $\nu$ is an
equilibrium state of $\frac{a}{a+b}\Psi$ and $\mu$ is a conditional
equilibrium state of $\frac{1}{a}\Phi$ with respect to $\nu$. Since
$\frac{a}{a+b}\Psi\in \C_{aa}(Y,S)$ and satisfies the bounded
distortion property, by Proposition \ref{pro-1.2}, $\nu$ is unique
and it satisfies the Gibbs property:
$$
\nu(J)\approx \exp\left(-nP\left(S,
\frac{a}{a+b}\Psi\right)\right)
\psi(J)^{\frac{a}{a+b}}=\exp\left(\frac{-nP^\ba(T,\Phi)}{a+b}\right)
\psi(J)^{\frac{a}{a+b}}
$$
for $n\in \N$ and $J\in \D^n$. This proves \eqref{e-2.2'}. Since
$\nu$ is quasi-Bernoulli, applying Proposition \ref{pro-1.4} to the
potential $\frac{1}{a}\Phi$, we see that $\mu$ is unique and
satisfies the Gibbs property:
$$
\mu(I)\approx \phi(I)^{\frac{1}{a}}\nu(\pi I)/\psi(\pi I)\approx
\exp\left( \frac{-nP^{\ba}(T,\Phi)}{a+b} \right)
\frac{\phi(I)^{1/a}} {\psi(\pi I)^{b/(a+b)}}
$$
for $n\in \N$ and $I\in \A^n$. This proves \eqref{e-2.2}. Note that
\eqref{e-2.3} follows directly from  \eqref{e-2.2} and
\eqref{e-2.2'}. This finishes the proof of Theorem \ref{thm-0.2}.
\end{proof}

\subsection{The proof of Theorem \ref{thm-0.3}}\label{S5} To prove
Theorem \ref{thm-0.3}, we need the following result which is just based on classical convex analysis.
\begin{pro}[\cite{FeHu09a}, Proposition 2.3]
\label{pro-2.3} Let $Z$ be a compact convex subset of a topological
vector space which satisfies the first axiom of countability (i.e.,
there is a countable base at each point)  and $U\subseteq \R^d$ a
non-empty open set. Suppose $f:\; U\times Z\to \R\cup\{-\infty\}$ is
a map satisfying the following  conditions:
\begin{itemize}
\item[(i)] $f(\bq,z)$ is convex in $\bq$;
\item[(ii)] $f(\bq,z)$ is affine in $z$;
\item[(iii)] $f$ is upper semi-continuous over $U\times Z$;
\item[(iv)] $g(\bq):=\sup_{z\in Z}f(\bq,z)>-\infty$ for any $\bq\in U$.
\end{itemize}
For each $\bq\in U$, denote $ \I(\bq):=\{z\in Z:\;
f(\bq,z)=g(\bq)\}. $ Then
\begin{equation*}
\label{e-1504}
\partial g(\bq)=\bigcup_{z\in \I(\bq)}\partial f(\bq,z),
\end{equation*}
where $\partial f(\bq,z)$ denotes the subdifferential of
$f(\cdot,z)$ at $\bq$.
\end{pro}

\begin{proof}[Proof of Theorem \ref{thm-0.3}]
 In Proposition \ref{pro-2.3}, we let $U=\R^d$,
$Z=\M(X,T)$, and define $f:\; U\times Z\to \R$ by
$$f(\bq, \mu)=\sum_{i=1}^d q_i(\Phi_i)_*(\mu)+ah_\mu(T)+bh_{\mu\circ
\pi^{-1}}(S),\qquad \bq=(q_1,\ldots,q_d).
$$
Set $g(\bq)=\sup_{z\in Z}f(\bq,z)=P^\ba(T,\sum_{i=1}^dq_i\Phi_i)$.
 Since $\Phi_i\in \C_{aa}(X,T)$, $\mu\mapsto
(\Phi_i)_*(\mu)$ is continuous on $\M(X,T)$ (see Lemma \ref{lem-2.1}(ii)). Thus, $f$ and $g$ satisfy the assumptions (i)-(iv) in
Proposition \ref{pro-2.3}. However by Theorem \ref{thm-0.2},
$\I(\bq)=\{\mu_\bq\}$ is a  singleton for each $\bq\in \R^d$. By
Proposition \ref{pro-2.3}, $\nabla
g(\bq)=((\Phi_1)_*(\mu_\bq),\ldots,(\Phi_d)_*(\mu_\bq))$. Since $g$
is convex and differentiable on $\R^d$, it is $C^1$ on $\R^d$ (see,
e.g.  \cite[Corollary 25.5.1]{Roc-book}). This finishes the proof of Theorem
\ref{thm-0.3}.

\end{proof}
\subsection{The proofs of Theorems \ref{thm-0.4} and \ref{thm-0.5}}
\label{S6}

\begin{proof}[Proof of Theorem \ref{thm-0.4}] Fix $n\in \N$. Denote by
$\Omega_n$ the collection of probability vectors
$\bp=(p(\omega))_{\omega\in \A^n}$ in $\R^{\A^n}$ satisfying
$$
\sum_{\varepsilon\in \A} p(\varepsilon x_2\ldots x_n )=\sum_{\varepsilon\in
\A}p(x_2\ldots x_n\varepsilon) \quad \mbox { for any word }x_2\ldots
x_n\in \A^{n-1}.
$$
It is clear that $\Omega_n$ is a convex compact subset of
$\R^{\A^n}$. In fact, $\Omega_n$ is the image of the following map
$$\eta\in \M(X,T)\mapsto (\eta(I))_{I\in \A^n}.$$
(cf. \cite[p. 232]{FFW01}).
Define a function $f:\; \Omega_n\to \R$ by
\begin{equation}
\label{e-3.2} f(\bp)=\sup\{ah_\eta(T)+bh_{\eta\circ \pi^{-1}}(S):\;
\eta\in \M(X,T): \;  (\eta(I))_{I\in \A^n}=\bp\}.
\end{equation}
The following properties of $f$ can be checked directly.
\begin{lem}
The map $f:\Omega_n\to \R$ is concave, bounded  and upper semi-continuous.
\end{lem}
Extend $f$ to a function on $\R^{\A^n}$ by
$$f(\bp)=-\infty\quad\mbox{  for }\bp\in \R^{\A^n}
\backslash \Omega_n$$ and define $f^*:\R^{\A^n}\to \R$ by
\begin{equation}
\label{e-3.3'} f^*(\bq)=\sup\left\{f(\bp)+\bp\cdot\bq:\;
\bp\in \R^{\A^n}\right\} =\sup\left\{f(\bp)+\bp\cdot\bq:\;
\bp\in \Omega_n\right\},
\end{equation}
where $\bp\cdot\bq$ denotes the standard inner product of $\bp$ and $\bq$ in $\R^{\A^n}$. Since $f$ is a bounded upper semi-continuous concave
function on $\Omega_n$,  we obtain
\begin{equation}
\label{e-3.3} f(\bp)=\inf \left\{ f^*(\bq)- \bp\cdot\bq:\;
\bq\in \R^{\A^n} \right\},\qquad \bp\in \Omega_n
\end{equation}
by using the  duality principle in convex analysis
(cf. \cite[Theorem 12.2]{Roc-book}). By (\ref{e-3.2}) and (\ref{e-3.3'}), we have
\begin{lem}
\label{lem-3.3} For $\bq=(q(I))_{I\in \A^n}\in \R^{\A^n}$,
\begin{equation*}
\begin{split}
f^*(\bq)&=\sup\left\{
\left(\sum_{I\in \A^n} q(I)\int \chi_{[I]} \; d\eta\right)+ah_\eta(T)+bh_{\eta\circ \pi^{-1}}(S):\; \eta\in \M(X,T)\right\} \\
&=P^{\ba}\left(T,\sum_{\omega\in \A^n}q(I)
\Phi_{I}\right),
\end{split}
\end{equation*}
where $\chi_{[I]}$ denotes the indicator function of
$[I]$, and $\Phi_{I}$ denotes the additive potential
$\left(\sum_{i=0}^{m-1}\chi_{[I]}(T^ix)\right)_{m=1}^\infty$.
Furthermore denote by $\mu$ the $\ba$-weighted equilibrium state of\\
$\sum_{I\in \A^n}q(I) \Phi_{I}$ and let
$\bp=(\mu(I))_{I\in \A^n}$. Then $\bp\in
\mbox{ri}(\Omega_n)$  and $f(\bp)=ah_\mu(T)+bh_{\mu\circ
\pi^{-1}}(S)$, where $\mbox{ri}(A)$ denotes the relative interior of a convex set $A$.
\end{lem}
By  Lemma \ref{lem-3.3} and Theorem \ref{thm-0.3}, $f^*$ is
differentiable on $\R^{\A^n}$. Hence by Corollary 26.4.1 in
\cite{Roc-book} and (\ref{e-3.3}), for any $\bp\in
\mbox{ri}(\Omega_n)$, there
exists  $\bq\in \R^{\A^n}$ such that
$$
\nabla f^*(\bq)=\bp.
$$
 It is easy to check that  $\mbox{ri}(\Omega_n)$ consists of the strictly positive vectors in $\Omega_n$.
 However, by Lemma \ref{lem-3.3} and Theorem \ref{thm-0.3},
$$
\nabla f^*(\bq)=\left(\mu(I)\right)_{I\in \A^n},
$$
where $\mu=\mu_{\bq}$ is the $\ba$-weighted equilibrium state of
$\sum_{I\in \A^n}q(I) \Phi_{I}$. By Theorem
\ref{thm-0.2}, $\mu_{\bq}$ is quasi-Bernoulli. Thus for each
positive vector $\bp$ in $\Omega_n$, there exists a quasi-Bernoulli
measure $\mu_{\bq}$ such that
$\left(\mu_\bq(I)\right)_{I\in \A^n}=\bp$. By  Lemma
\ref{lem-3.3}, we do have
\begin{equation*}
\begin{split}
&ah_{\mu_{\bq}}(T)+bh_{\mu_{\bq}\circ \pi^{-1}}(S)=f(\bp)\\
&\mbox{}\qquad = \sup\{ah_\eta(T)+bh_{\eta\circ \pi^{-1}}(S):\;
\eta\in \M(X,T): \;  (\eta(I))_{I\in \A^n}=\bp\}.
\end{split}
\end{equation*}
Furthermore, the measure $\mu$ which attains the supremum  is
unique, because each such a measure is a $\ba$-weighted equilibrium
state of $\sum_{I\in \A^n}q(I) \Phi_{I}$. This
finishes the proof of Theorem \ref{thm-0.4}.
\end{proof}

\begin{proof}[Proof of Theorem \ref{thm-0.5}] First assume that
$\eta$ is fully supported. Let $\mu_n= \mu(\ba, \eta,n)$ as in
Theorem \ref{thm-0.4}. Then the sequence $(\mu_n)_{n=1}^\infty$ is
desired in Theorem \ref{thm-0.5}.

Now consider the general case. Let $\eta_n=(1-1/n)\eta+(1/n)\eta_0$,
where $\eta_0$ denotes the Parry measure on $X$. Clearly, $\eta_n$
is fully supported. Denote $\mu_n'=\mu(\ba, \eta_n,n)$. Then
$(\mu_n')_{n=1}^\infty$ is desired.
\end{proof}

\section{Higher dimensional weighted thermodynamic formalism}
\label{S7}

 In this section, we present the higher dimensional
versions of our main results. Since the proofs are  essentially
identical to those in the two dimensional case, we just omit them.

Let $k\geq 2$. Assume that $(X_i, T_i)$ ($i=1,\ldots, k$) are
subshifts over finite alphabets $\A_i$ such that $X_{i+1}$ is a factor of $X_i$ with a one-block
factor map $\pi_i: \;X_i\to X_{i+1}$ for $i=1,\ldots, k-1$. For
convenience, we use $\pi_0$ to denote the identity map on $X_1$.
Define $\tau_i:\;X_1\to X_{i+1}$ by $\tau_i=\pi_i\circ\pi_{i-1}\circ
\cdots \circ \pi_0$ for $i=0,1,\ldots,k-1$.

Let $\ba=(a_1,\ldots, a_k)\in \R^k$ so that  $a_1> 0$ and $a_{i}\geq
0$ for $i>1$.  For  $\Phi\in \C_{ass}(X_1,T_1)$. We  define the {\it
$\ba$-weighted topological pressure of $\Phi$}  as
$$
P^{\ba}(T_1,\Phi) =\sup\left\{ \Phi_*(\mu)+h^\ba_\mu(T_1):\; \mu\in \M(X_1,T_1)\right\},
$$
where $h^\ba_\mu(T_1)$ is the $\ba$-weighted topological entropy defined as
$$
h^\ba_\mu(T_1)=\sum_{i=1}^{k}a_ih_{\mu\circ \tau_{i-1}^{-1}}(T_i).
$$
Clearly the supremum is attainable. Each measure $\mu$ which attains
the supremum  is called an $\ba$-weighted equilibrium state of
$\Phi$.

For $i=1,\ldots, k-1$, we define $\theta_i:\; \C_{ass}(X_i, T_i)\to
\C_{ass}(X_{i+1}, T_{i+1})$ by $(\log \phi_n)_{n=1}^\infty\mapsto
(\log \psi_n)_{n=1}^\infty$, where
$$
\psi_n(y)=\left(\sum_{I\in \A_{i}^n:\; [I]\cap \pi_{i}^{-1}(y)\neq
\emptyset} \sup_{x\in [I]\cap
\pi_{i}^{-1}(y)}\phi_n(x)^{1/A_i}\right)^{A_i}
$$
for $y \in X_{i+1}$, with $A_i=a_1+\cdots+a_i$. In particular, let ${\mathcal S}_{ass}$ denote
the collection of asymptotically sub-additive additive (scalar) sequences $(\log
c_n)_{n=1}^\infty$ (a sequence $(\log
c_n)_{n=1}^\infty$, where $c_n\geq 0$, is called {\it asymptotically sub-additive} if, for any $\varepsilon>0$, there exists a sequence $(d_n)_{n=1}^\infty$, so that $0\leq d_{n+m}\leq d_nd_m$ and $\limsup_{n\to \infty}\frac{1}{n}|\log c_n-\log d_n|<\varepsilon$).  Let $ \theta_k:\; \C_{ass}(X_k,T_k)\to {\mathcal
S}_{ass} $
 be defined as $(\log \phi_n)_{n=1}^\infty\mapsto (\log c_n)_{n=1}^\infty$,
where
$$
c_n=\left(\sum_{I\in \A_{k}^n}\sup_{x\in
[I]}\phi_n(x)^{1/{A_k}}\right)^{A_k}.
$$

As an extension of Theorem \ref{thm-0.1}, we have
\begin{thm}\label{thm-0.1'}
\begin{itemize}
\item[(i)] $P^{\ba}(T_1,\Phi)=
\lim_{n\to \infty}(1/n)\log c_n$, where $(c_n)_{n=1}^\infty=\theta_k\circ\cdots\circ \theta_1(\Phi)$.
\item[(ii)] For any $1\leq i\leq k-1$,
$P^{\ba}(T_1,\Phi)=P^{(\sum_{j=1}^i a_j,\; a_{i+1},\ldots,
a_k)}(T_{i+1},\; \theta_{i}\circ \cdots \circ\theta_1(\Phi))$.
\item[(iii)] $\mu\in \M(X_1,T_1)$ is an $\ba$-weighted equilibrium
state of $\Phi$ if and only if $\mu\circ \tau_{k-1}^{-1}$ is an
equilibrium state of $\frac{\theta_{k-1}\circ \cdots \circ
\theta_1(\Phi)}{a_1+\cdots+a_k}$ and, for $i=k-2,k-3,\ldots,0$,
$\mu\circ \tau_{i}^{-1}$ is a conditional equilibrium state of
$\frac{\theta_i\circ \cdots\circ
\theta_1(\Phi)}{a_1+\cdots+a_{i+1}}$ with respective to $\mu\circ
\tau_{i+1}^{-1}$.
\end{itemize}
\end{thm}

In the remaining part of this section,  we assume that $X_i$ is the
full shift over  $\A_i$ for each $i\in \{1,\ldots, k\}$. For
$i=1,\ldots, k-1$, we redefine $\theta_i:\; \C_{asa}(X_i, T_i)\to
\C_{asa}(X_{i+1}, T_{i+1})$ by $(\log \phi_n)_{n=1}^\infty\mapsto
(\log \psi_n)_{n=1}^\infty$, where
\begin{equation*}
\label{e-theta}
\psi_n(y)=\left(\sum_{I\in \A_{i}^n:\; [I]\cap \pi_{i}^{-1}(y)\neq
\emptyset} \sup_{x\in
[I]}\phi_n(x)^{1/A_i}\right)^{A_i}
\end{equation*}
for $y \in X_{i+1}$. In particular, let ${\mathcal S}_{asa}$ denote
the collection of asymptotically additive  (scalar) sequences $(\log
c_n)_{n=1}^\infty$.  Let $ \theta_k:\; \C_{asa}(X_k,T_k)\to {\mathcal
S}_{asa} $
 be defined as $(\log \phi_n)_{n=1}^\infty\mapsto (\log c_n)_{n=1}^\infty$,
where
\begin{equation*}
\label{e-c}
c_n=\left(\sum_{I\in \A_{k}^n}\sup_{x\in
[I]}\phi_n(x)^{1/A_k}\right)^{A_k}.
\end{equation*}

For  $\Phi=(\log \phi_n)_{n=1}^\infty\in \C_{asa}(X_1,T_1)$, write
  \begin{equation}
  \label{e-not}
  \begin{split}
    &(\log \phi_n^{(i)})_{n=1}^\infty:=\theta_i\circ\cdots\circ \theta_1(\Phi),\quad i=1,\ldots, k,\\
    &(\log \phi_n^{(0)})_{n=1}^\infty:=(\log \phi_n)_{n=1}^\infty \mbox{ and }\\
    &\phi^{(i)}(J):=\sup\{\phi_n^{(i)}(y):\; y\in [J]\}
  \end{split}
  \end{equation}
 for any $n$-th cylinder $[J]\subset X_{i+1}$, $i=0,\ldots, k-1$. Then, we define the {\it $\ba$-weighted potential associated with $\Phi$}  by
 \begin{equation}\label{phia}
\Phi^\ba= (\log \phi^\ba_n)_{n=1}^\infty,\text{ where } \phi^\ba_n(x)=\phi^{(0)}(x_{|n})^{1/A_1}\prod_{i=1}^{k-1} \phi^{(i)}(\tau_i(x_{|n}))^{1/A_{i+1}- 1/A_i},
\end{equation}
where $A_i=a_1+\cdots +a_i$. Since there exists a sequence $(g^{(p)})_{p\ge 1}$ of H\"older potentials such that $\lim_{p\to 0}\limsup_{n\to\infty}\|\Phi_n-S_ng^{(p)}\|_\infty/n=0$ (see Lemma~\ref{lem-2.1}(iii)), it is easily seen that all the potentials $(\log \phi^{(i)}(\tau_{i-1}(\cdot_{|n}))_{n=1}^\infty$ and $(\log \phi^\ba_n)_{n=1}^\infty$ belong to $\C_{asa}(X,T)$.

As an analogue of Theorems \ref{thm-0.2}-\ref{thm-0.5}, we have
\begin{thm}\label{thm-0.2'}
\begin{itemize}
\item[(i)]
Let $\Phi=(\log \phi_n)_{n=1}^\infty\in \C_{asa}(X_1,T_1)$.  Then $$P^{\ba}(T_1,\Phi)=\lim_{n\to \infty}(1/n)\log c_n,$$ where $(\log c_n)_{n=1}^\infty=\theta_k\circ\cdots\circ \theta_1(\Phi)$.
\item[(ii)]
Assume   $\Phi\in \C_{aa}(X_1,T_1)$ and $\Phi$ has the bounded distortion property. Then
there is a unique $\ba$-weighted equilibrium state $\mu$ of $\Phi$. The measure $\mu$ is fully supported and quasi-Bernoulli, and it satisfies the following Gibbs property
\begin{equation}\label{defmu4'}
\mu(I)\approx\exp \Big (\frac{-nP}{A_k}\Big) \phi_n^\ba(I), \quad I\in \A_1^n,
\end{equation}
where $P=P^{\ba}(T_1,\Phi)$.
Consequently, for $i=2,\ldots,k$,
\begin{equation}\label{defmu'}
\mu_i(\tau_{i-1}I)\approx \exp \Big (\frac{-nP}{A_k}\Big) \phi^{(i-1)}(\tau_{i-1}I)^{1/A_i}\prod_{j=i}^{k-1} \phi^{(j)}(\tau_jI)^{1/A_{j+1}- 1/A_j},\quad I\in \A_1^n,
\end{equation}
where $\mu_i:=\mu\circ \tau_{i-1}^{-1}$. Furthermore,
\begin{equation*}
\phi_n(x)\exp(-nP)\approx \prod_{i=1}^{k} \mu_i (\tau_{i-1}x_{|n})^{a_{i}} \quad \mbox{ for }x\in X_1,\; n\geq 1,
\end{equation*}

\end{itemize}
\end{thm}

A Borel probability measure $\mu$ (not necessarily invariant) on $X$ satisfying \eqref{defmu4'} is called {\it an $\ba$-weighted Gibbs measure for $\Phi$}.

\begin{thm}
\label{thm-0.3'} Let $\Phi_1,\ldots, \Phi_d\in \C_{aa}(X,T)$ satisfy
the bounded distortion property. Then the map $Q:\R^d \to \R$
defined as
$$(q_1,\ldots, q_d)\mapsto P^{\ba}\left(T,\sum_{i=1}^dq_i\Phi_i\right),$$
is $C^1$ over $\R^d$ with
$$\nabla Q(q_1,\ldots, q_d)=((\Phi_1)_*(\mu_\bq),\ldots,
(\Phi_d)_*(\mu_\bq)),$$
 where
$\mu_{\bq}$ is the unique $\ba$-weighted equilibrium state of
$\sum_{i=1}^dq_i\Phi_i$.
\end{thm}

\begin{thm}\label{thm-0.4'}
 For each fully supported measure $\eta\in \M(X_1,T_1)$ and each $n\in \N$, there is a unique measure
 $\mu=\mu(\ba, \eta,n)$ in $\M(X_1,T_1)$ attaining the following supremum
$$
\sup\left\{h^\ba_\mu(T_1):\;
\mu(I)=\eta(I) \mbox{ for all $n$-th cylinder }
[I]\in X_1\right\}.
$$
Furthermore $\mu(\ba, \eta,n)$ is the $\ba$-weighted equilibrium
state of certain $\Phi\in \C_{aa}(X_1,T_1)$ with the bounded
distortion property, and hence $\mu(\ba, \eta,n)$ is a  fully
supported quasi-Bernoulli measure on $X_1$.
\end{thm}

\begin{thm}\label{thm-0.5'}
For any $\eta\in \M(X_1,T_1)$, there exists $(\mu_n)_{n=1}^\infty\subset
\M(X_1,T_1)$ converging  to $\eta$ in the weak-star topology such
that for each $n$,   $\mu_n$ is quasi-Bernoulli  and
\begin{equation*}
h^\ba_{\mu_n}(T_1)\ge h^\ba_{\mu}(T_1).
\end{equation*}
Furthermore,
$$
\lim_{n\to \infty}h^\ba_{\mu_n}(T_1)=h^\ba_{\mu}(T_1).
$$
\end{thm}

\begin{rem}\label{simultapprox2}
{\rm
If we take $\ba=(1,\ldots,1)$, due to the upper semi-continuity of the entropy, for any $\mu\in\M(X,T)$, Theorem~\ref{thm-0.5'} yields a sequence of quasi-Bernoulli measures $(\mu_n)_{n=1}^\infty$ which converges to $\mu$ in the weak-star topology,  such that we have both $\lim_{n\to\infty} h_{\mu_n}(T)= h_{\mu}(T)$ and  $\lim_{n\to\infty} h_{\mu_n\circ\pi^{-1}}(S)=  h_{\mu\circ\pi^{-1}}(S)$. Moreover, one can deduce from Theorem~\ref{thm-0.2'} that for any $\ba=(a_1,\ldots,a_k)$ with $a_1>0$ and $a_i\geq 0$ for $i\geq 2$,  each  invariant quasi-Bernoulli measure is the $\ba$-weighted equilibrium state of some almost additive  potential satisfying the bounded distortion property.
}
\end{rem}

\begin{de}
{\rm We say that two almost additive potentials $\Phi=(\log \phi_n)_{n=1}^\infty$ and $\Psi=(\log \psi_n)_{n=1}^\infty$ are {\it cohomologous} if $\sup_n\|\log \phi_n-\log \psi_n\|_\infty<\infty$. If there exists $C\in\mathbb{R}$ such that $\log \psi_n=C n$, we say that $\Phi$ is {\it cohomologous to a constant}.
}
\end{de}
The following proposition is a direct consequence of Theorem~\ref{thm-0.2'}.
\begin{pro}\label{cohom}
Let $\Phi,\;\Psi\in \C_{aa}(X,T)$ satisfy the bounded distortion property. Then,  $\Phi$ and $\Psi$ share the same $\ba$-weighted equilibrium state if and only if $\Phi-\Psi$ is cohomologous to a constant.
\end{pro}
Next theorem is reminiscent from Sections 4.6 and 4.7 of \cite{Rue78}.
\begin{thm}\label{strictconv}
Let $\Phi_1,\ldots, \Phi_d\in \C_{aa}(X,T)$ satisfy
the bounded distortion property. Let $V$ be the vector subspace of those $\bq$ such that $\sum_{i=1}^dq_i\Phi_i$ is  cohomolohous to a constant. The map $Q$ defined in Theorem~\ref{thm-0.3'} is strictly convex if and only if $V=\{{\bf 0}\}$. Moreover, $Q$ is affine on any affine subspace of $\mathbb{R}^d$ parallel to $V$. In particular, if $d=1$ and $Q$ is not strictly convex, it is affine.
\end{thm}
An immediate corollary is
\begin{cor}
Let $\Phi_1,\ldots, \Phi_d\in \C_{aa}(X,T)$ satisfy
the bounded distortion property. Let ${\bf \Phi}=(\Phi_1,\dots,\Phi_d)$. The convex set $ \big \{\big( (\Phi_1)_*(\mu),\dots, (\Phi_d)_*(\mu)\big):\mu\in\M(X,T) \big\}$ is reduced to a singleton if and only if each $\Phi_i$ is cohomologous to a constant.
\end{cor}

\begin{proof}[Proof of Proposition \ref{cohom}]
 Suppose that $Q$ is affine on a non-trivial segment $[\bq, \bq']$. For every $t\in [0,1]$ we have
\begin{eqnarray*}Q(\bq+t(\bq'-\bq))&=&Q(\bq)+t \nabla Q(\bq)\cdot(\bq'-\bq)\\
&=&
Q(\bq)+t \sum_{i=1}^d(q_i'-q_i)(\Phi_i)_*(\mu_\bq).
\end{eqnarray*}
 Since $Q({\bf q})= \sum_{i=1}^dq_i(\Phi_i)_*(\mu_\bq)+\sum_{i=1}^ka_ih_{\mu\circ\tau_{i-1}^{-1}}(T_i)$, we have
$$
Q(\bq+t(\bq'-\bq))=\sum_{i=1}^d(q_i+t (q_i'-q_i))(\Phi_i)_*(\mu_\bq)+ \sum_{i=1}^ka_ih_{\mu\circ\tau_{i-1}^{-1}}(T_i).
$$
Consequently, $\mu_\bq$ is the unique $\ba$-weighted equilibrium state of $\sum_{i=1}^d(q_i+t (q_i'-q_i))\Phi_i$, for each $t\in [0,1]$. Due to Proposition~\ref{cohom}, this implies that $\sum_{i=1}^d(q'_i-q_i)\Phi_i$ is cohomologous to a constant, hence $V\neq\{{\bf 0}\}$.

Conversely, assume that $V\neq\{{\bf 0}\}$. Then, the same argument as above can be used to prove that $Q$ is affine on any affine subspace of $\mathbb{R}^d$ parallel to $V$.
\end{proof}

\section{Multifractal analysis on higher dimensional self-affine symbolic spaces}\label{multifractalanalysis}

Let $k\geq 2$. Assume that $(X_i, T_i)$ ($i=1,\ldots,k$) are  full shifts over $\A_i$ such that $X_{i+1}$ is a factor of $X_i$ with a one-block
factor map $\pi_i: \;X_i\to X_{i+1}$ for $i=1,\ldots, k-1$. For
convenience, we use $\pi_0$ to denote the identity map on $X_1$.
Define $\tau_i:\;X_1\to X_{i+1}$ by $\tau_i=\pi_i\circ\pi_{i-1}\circ
\cdots \circ \pi_0$ for $i=0,1,\ldots,k-1$. We simply write $(X,T)$ for $(X_1,T_1)$.

For $x=(x_i)_{i=1}^\infty \in X$ and $n\ge 1$, $x_{|n}$ denotes  the word $x_1\cdots x_n$.

We endow the set $X$ with a ``self-affine'' metric as follows. We fix $\ba=(a_1,\ldots, a_k)\in \R^k$ with $a_1> 0$ and $a_{i}\geq 0$ for $i>1$, and we define  the ultrametric distance
$$
d_\ba(x,y)=\max\big (e^{-|\tau_{i-1}(x)\land \tau_{i-1}(y)|/(a_1+\cdots+a_{i})}:1\le i\le k\big).
$$
For $1\le i\le k$ and $n\in \N$,  let
$$\ell_i(n)=\min \{p\in\N: p\ge (a_1+\cdots+a_i)n/a_1\},$$
 and by convention set $\ell_0(n)=0$. It is easy to check that

\begin{lem}\label{lem-simple}
In $(X,d_\ba)$,  the closed ball centered at $x$ of radius $e^{-n/a_1}$ is given by
\begin{equation*}\label{Bxr}
B(x,e^{-n/a_1})= \left\{y\in X:  \tau_{i-1}(y)\in \tau_{i-1}({x}_{|\ell_i(n)}) \mbox{ for all } 1\leq i\leq k\right\}.
\end{equation*}
\end{lem}

The following result  estimates the value  of an $\ba$-weighted Gibbs measure on a ball in  $(X,d_\ba)$.

\begin{lem}
\label{lem-gibbs}
Let $\Phi=(\log \phi_n)_{n=1}^\infty\in \C_{aa}(X,T)$ satisfy the bounded distortion property. Let $\mu$ denote the $\ba$-weighted Gibbs measure of $\Phi$.
Then we have the following estimate:
$$
 \mu(B(x,e^{-n/a_1}))\approx \exp \Big (\frac{-n P^\ba(T, \Phi)}{a_1}\Big)\phi_n(x)^{1/a_1} \prod_{j=1}^{k-1}\frac{\phi^{(j)}(\tau_j(x_{|\ell_{j+1}(n)}))^{1/A_{j+1}}}{\phi^{(j)}(\tau_j(x_{|\ell_{j}(n)}))^{1/A_{j}}},
$$
where $\phi^{(j)}$,  $j=0,\ldots,k-1$, are defined as in \eqref{e-not}, and $A_j=a_1+\cdots+a_j$.
\end{lem}
\begin{proof}
 Let $x=(x_i)_{i=1}^\infty\in X$ and  $n\ge 1$.  For $i=1,\ldots, k$, write $I_i=x_{\ell_{i-1}(n)+1}\cdots x_{\ell_{k}(n)}$. Let $B$ denote $B(x,e^{-n/a_1})$.
 By Lemma \ref{lem-simple},
 $B=\{y: \forall\ 1\le i\le k, \ \tau_{i-1}(y)\in \tau_{i-1}([I_{1}\dots I_{i}])\}$.  Since $\mu$ is quasi-Bernoulli (cf. Theorem \ref{thm-0.2'}(ii)),
 we have $\mu(B)\approx \prod_{i=1}^{k}\mu_{i}  (\tau_{i-1}I_i)$, where $\mu_i=\mu\circ \tau_{i-1}^{-1}$. Let us transform this expression by using \eqref{defmu'}. Since each word $I_i$ is of length  $\ell_i(n)-\ell_{i-1}(n)$ and by construction $\ell_{k}(n)/A_k-n/a_1=O(1/n)$,  \eqref{defmu'} yields
\begin{eqnarray*}
 \mu(B)&\approx&\exp \Big (\frac{-\ell_{k}(n) P^\ba(T, \Phi)}{A_k}\Big)\prod_{i=1}^{k} \phi^{(i-1)}(\tau_{i-1}I_i)^{1/A_i}\prod_{j=i}^{k-1} \phi^{(j)}(\tau_jI_i)^{1/A_{j+1}- 1/A_j}
 \\
 &\approx&\exp \Big (\frac{-n P^\ba(T, \Phi)}{a_1}\Big) \Big (\prod_{i=0}^{k-1} \phi^{(i-1)}(\tau_{i-1}I_i)^{1/A_i}\Big ) \prod_{j=1}^{k-1} \prod_{i=1}^j\phi^{(j)}(\tau_jI_i)^{1/A_{j+1}- 1/A_j}
\\
& \approx & \exp \Big (\frac{-n P^\ba(T, \Phi)}{a_1}\Big) \phi^{(0)}(I_1)^{1/a_1} \prod_{j=1}^{k-1}\frac{\phi^{(j)}(\tau_j(I_1\cdots I_{j+1}))^{1/A_{j+1}}}{\phi^{(j)}(\tau_j(I_1\cdots I_{j}))^{1/A_{j}}}\\
&\approx &\exp \Big (\frac{-n P^\ba(T, \Phi)}{a_1}\Big)\phi_n(x)^{1/a_1} \prod_{j=1}^{k-1}\frac{\phi^{(j)}(\tau_jx_{|\ell_{j+1}(n)})^{1/A_{j+1}}}{\phi^{(j)}(\tau_jx_{|\ell_{j}(n)})^{1/A_{j}}}.
\end{eqnarray*}
This finishes the proof of the lemma.
\end{proof}

Recall that the weighted entropy of $\mu\in \M(X,T)$ has been defined in Section~\ref{S7} as $h^\ba_{\mu}(T)=\sum_{i=1}^ka_ih_{\mu\circ\tau_{i-1}^{-1}}(T_i)$. The following Ledrappier-Young type formula was  proved by Kenyon and Peres in \cite[ Lemma 3.1]{KePe96} under a slight different setting.
 \begin{pro}\label{dimmu1}
Suppose that $\mu\in \mathcal{M}(X,T)$ is ergodic. Then we have
$$
\dim_H\mu=h^\ba_{\mu}(T).
$$
\end{pro}

\subsection{Multifractal analysis of asymptotically additive potentials}\label{results}
Recall that the generic set $\G(\mu)$ of a measure $\mu\in\M(X,T)$  has been defined in \eqref{generic}, and that an equivalent definition invoking asymptotically additive potentials is given in Remark \ref{re-2.2}. We have the following high dimensional extension of Theorem~\ref{thm-0.6}.
\begin{thm}\label{thm-0.6'}
Let $\mu\in \M(X,T)$. We have $\G(\mu)\neq\emptyset$ and $\dim_H \mathcal{G}(\mu)=h^\ba_{\mu}(T)$.
\end{thm}

The proof of Theorem \ref{thm-0.6'} will be given in Sect.~\ref{s-5.4}.
Next we consider level sets associated with Birkhoff averages of asymptotically  additive potentials on $X$.

For ${\bf\Phi}=(\Phi_1,\dots,\Phi_d)\in \C_{asa}(X,T)^d$,
where $\Phi_i=(\log \phi_{n,i})_{n=1}^\infty=:(\Phi_{n,i})_{n=1}^\infty$, and
 $\alpha=(\alpha_1,\ldots,\alpha_d)\in
\R^d$, define
\begin{equation}
\label{e-alpha}
E_{{\bf \Phi}}(\alpha)=\Big\{x\in X: \lim_{n\to\infty} \frac {
\Phi_{n,i}(x)}{n}=\alpha_i \mbox{ for } 1\leq i\leq d\Big\}.
\end{equation}
Denote ${\bf \Phi}_n(x)=(\Phi_{n,1}(x),\ldots, \Phi_{n,d}(x))$. Then the set in the right hand side of \eqref{e-alpha} can be simply written as  $\Big\{x\in X: \lim_{n\to\infty} \frac {
{\bf\Phi}_{n}(x)}{n}=\alpha\Big\}$.
For $\mu\in \M(X,T)$, write ${\bf
\Phi}_*(\mu)=((\Phi_1)_*(\mu),\ldots, (\Phi_d)_*(\mu))$ and define $L_{\bf\Phi}=\big \{{\bf \Phi}_*(\mu):\mu\in\M (X,T)\big\}$.

Let $\{{\bf \Phi}^{(j)}\}_{1\le j\le r}$ be a family of elements of  $\C_{asa}(X,T)^d$.  Let $\bc=(c_1,\dots,c_r)$ be a real vector with positive entries. For $\alpha\in
\R^d$, define
$$
E_{\{{\bf \Phi}^{(j)}\},\bc}(\alpha)=\Big\{x\in X: \lim_{n\to\infty} \sum_{j=1}^r\frac {{\bf \Phi}^{(j)}_{\lfloor c_j n\rfloor }}{\lfloor c_jn\rfloor }(x)=\alpha\Big\},
$$
where $\lfloor y\rfloor$ stands for the integer part of $y\in \R$. It is clear that  $E_{\{{\bf\Phi}^{(j)}\},\bc}(\alpha)=E_{\{{\bf\Phi}^{(j)}\},\lambda\bc}(\alpha)$ for any $\lambda>0$, and in particular, $E_{\{{\bf\Phi}^{(j)}\},\bc}(\alpha)=E_{\bf\Phi}(\alpha)$ if $r=1$. It is remarkable that the Hausdorff dimension of the set $E_{\{{\bf \Phi}^{(j)}\},\bc}(\alpha)$ does not depend on~$\bc$ when $r\ge 2$, as shown in  the following result,  of which the proof will be given in Sect.~\ref{s-5.5}.
\begin{thm}\label{thm-0.7'} Let ${\bf \Phi}=\sum_{j=1}^r {\bf \Phi}^{(j)}$.
\begin{enumerate}
\item For $\alpha\in \R^d$, the following assertions are equivalent.
\begin{itemize}
\item[(i)] $\alpha\in L_{\bf\Phi}$;
\item [(ii)]$E_{\{{\bf \Phi}^{(j)}\},\bc}(\alpha)\neq\emptyset$;
\item [(iii)]$\inf \left\{P^{\ba}(T,  \bq\cdot {\bf \Phi})-\alpha\cdot
\bq:\; \bq\in \R^d\right\}\ge 0$;
\item [(iv)]$\inf \left\{P^{\ba}(T,  \bq\cdot {\bf \Phi})-\alpha\cdot
\bq:\; \bq\in \R^d\right\}>-\infty$;
\end{itemize}
Furthermore for $\alpha\in L_{\bf \Phi}$, we have
\begin{equation*}
\begin{split}
\dim_H E_{\{{\bf \Phi}^{(j)}\},\bc}(\alpha)&=\max\left\{h^\ba_{\mu}(T):\; \mu\in \M(X,T),\; {\bf \Phi}_*(\mu)=\alpha\right\}\\
&=\inf \left\{P^{\ba}(T,  \bq\cdot {\bf \Phi})-\alpha\cdot
\bq:\; \bq\in \R^d\right\}.
\end{split}
\end{equation*}
\item Suppose that $L_{\bf \Phi}$ is not  a singleton. Then the set $X\setminus \bigcup_{\alpha\in L_{\bf \Phi}}E_{\{{\bf \Phi}^{(j)}\},\bc}(\alpha)$ is of full Hausdorff dimension.
\end{enumerate}
\end{thm}
\begin{rem}\label{dimX}
If we take $r=1$ and ${\bf\Phi}=0$, we find that the Hausdorff dimension of $(X,d_\ba)$ is $P^\ba(T,0)$. This extends the result of \cite{KePe96} which holds for special choices of $\ba$.
\end{rem}


\begin{ex}
\label{ex-5.7}
{\rm
Generally, the level sets  $E_{\{{\bf \Phi}^{(j)}\},\bc}(\alpha)$ depend on ${\bf c}$. For example, let $X=\{0,1\}^\N$, and let $g\in C(X)$ be given by
$g(x)=x_1$ for $x=(x_i)_{i=1}^\infty\in X$.   Set $\Phi^{(1)}=(S_ng)_{n=1}^\infty$ and $\Phi^{(2)}=(-S_ng)_{n=1}^\infty$. Then
$E_{\{{\bf \Phi}^{(j)}\}_{j=1}^2, (1,1)}(0)=X$, however $E_{\{{\bf \Phi}^{(j)}\}_{j=1}^2, (1,2)}(0)\neq X$ (it is easy to check that $x=0^11^20^41^8\cdots 0^{2^{2n}}1^{2^{2n+1}}\cdots\not\in E_{\{{\bf \Phi}^{(j)}\}_{j=1}^2, (1,2)}(0)$).
}
\end{ex}

\subsection{Application to the multifractal analysis of $\ba$-weighted weak Gibbs measures}

As we have seen in Theorem~\ref{thm-0.2'}, $\ba$-weighted Gibbs measures are naturally associated with almost additive potentials satisfying the bounded distortion property; this extends the classical Gibbs measures. Now we show that the notion of weak Gibbs measure associated with a continuous potential defined on $X$ in the classical thermodynamic formalism \cite{Kes01} also has a natural extension in the $\ba$-weighted thermodynamical formalism.

\begin{de}
{\rm
Let $\Phi\in \C_{asa}(X,T)$. A fully supported Borel probability measure $\mu$ (not necessarily to be shift invariant) on $X$ is called an  {\it  $\ba$-weighted weak Gibbs measure
associated with $\Phi$} if
\begin{equation}\label{defmu4}
\mu(I)\approx_n\exp \Big (\frac{-nP}{A_k}\Big) \phi_n^\ba(I), \quad I\in \A^n,
\end{equation}
where
$P=P^\ba(T_1,\Phi)$, $A_k=a_1+\cdots+a_k$,  $\Phi^\ba=(\log \phi^\ba_n)\in \C_{asa}(X_1,T_1)$ is defined as in \eqref{phia}, and
$\approx_n$ means that  there exists a sequence of positive numbers $(\kappa_n)_{n=1}^\infty$ with $\lim_{n\to \infty}(1/n)\log \kappa_n=0$, such that
the ratio between the left and right hand sides of  $\approx_n$ lies in $(\kappa_n^{-1}, \kappa_n)$.
}

\end{de}

\begin{rem}
It is not hard to see that if $\mu$ satisfies (\ref{defmu4}), then for $i=2,\ldots,k$,
\begin{equation}\label{defmu2}
\mu_i(\tau_{i-1}I)\approx_n \exp \Big (\frac{-nP}{A_k}\Big) \phi^{(i-1)}(\tau_{i-1}I)^{1/A_i}\prod_{j=i}^{k-1} \phi^{(j)}(\tau_jI)^{1/A_{j+1}- 1/A_j},\quad I\in \A^n,
\end{equation}
where $\mu_i=\mu\circ \tau_{i-1}^{-1}$, and $\phi^{(j)}$,  $j=0,\ldots k-1$, are defined as in \eqref{e-not}. Furthermore, $\mu$ satisfies \eqref{defmu4} if and only if
\begin{equation*}
    \label{equ}
    \phi_n(x)\exp(-nP)\approx_n \prod_{i=1}^{k} \mu_i (\tau_{i-1}x_{|n})^{a_{i}}, \quad  \;x\in X,\; n\geq 1,
    \end{equation*}
\end{rem}

The following result, which will be proved in Sect.~\ref{s-5.6},  shows the existence of $\ba$-weighted weak Gibbs measure for any asymptotically additive potential on $X$.

\begin{thm}\label{weakGibbs}
Let $\Phi=(\log\phi_n)_{n=1}^\infty\in \C_{asa}(X,T)$. Then there exists at least an  $\ba$-weighted weak Gibbs measure $\mu$ associated with $\Phi$.

Furthermore, for each $1\le i\le k$, the potential ${\bf \Psi}^{(i)}_\mu:=\big (\log \mu_i(\tau_{i-1}(x_{|n}))\big )_{n=1}^\infty$ belongs to $\C_{asa}(X,T)$, and for every point $x=(x_i)_{i=1}^\infty\in X$ and  $B=B(x,e^{-n/a_1})$, we have
\begin{equation}\label{logmassball1}
\log \mu(B)={\bf \Psi}^{(1)}_{\mu,n}(x)+\sum_{i=2}^k {\bf \Psi}^{(i)}_{\mu,\ell_i(n)}(x)- {\bf \Psi}^{(i)}_{\mu,\ell_{i-1}(n)}(x)+c(x,n),
\end{equation}
where $(c(x,n))_{n\geq 1}$ is a sequence satisfying $\lim_{n\to\infty} c(x,n)/n=0$.  If moreover,  $\Phi\in\C_{aa}(X,T)$ and satisfies the bounded distortion property, then $c(x,n)$ can be taken bounded independently of $x$ and $n$, and \eqref{logmassball1} takes the form
\begin{equation}\label{massball1}
\mu(B)\approx \prod_{i=1}^{k}\mu_i  (\tau_{i-1}(I_{i})),
\end{equation}
where $I_i=x_{\ell_{i-1}(n)+1}\cdots x_{\ell_i(n)}$.
\end{thm}

\begin{rem}
{\rm
\begin{itemize}
\item
[(1)] We recover the usual weak Gibbs measures when $\ba=(1,0\dots,0)$ and $\Phi$ is the sequence of Birkhoff sums associated with a continuous potential over $X$ \cite{Yur97, Kes01}.

 \item
 [(2)] By using \eqref{defmu4} and \eqref{defmu2}, from any $(1,0,\dots,0)$-weighted weak Gibbs measure $\mu$ one can build an asymptotically additive potential of which $\mu$ is an $\ba$-weighted weak Gibbs measure.

 \end{itemize}
 }
\end{rem}

We have the following result on the multifractal analysis of $\ba$-weighted weak Gibbs measures.
\begin{thm}\label{thm-0.9'}
Let $\mu$ be an $\ba$-weighted weak Gibbs measure associated with some asymptotically additive potential. For $\alpha\in\mathbb{R}_+$ we define
$$
E_\mu(\alpha)=\Big\{x\in X: \lim_{r\to 0^+}\frac{\log
\mu(B(x,r))}{\log r}=\alpha\Big \}.
$$
Let ${\bf \Psi}_\mu=\sum_{i=1}^k a_i{\bf \Psi}^{(i)}_\mu$. Let $L_\mu=L_{-{\bf \Psi}_\mu}=\{-({\bf \Psi}_\mu)_*(\lambda):
\lambda\in\mathcal{M}(X,T)\}$. Then, for all $\alpha\geq 0$,
$E_\mu(\alpha)\neq\emptyset$ if and only if $\alpha\in L_\mu$.
For $\alpha\in L_\mu$, we have
\begin{equation*}
\begin{split}
\dim_HE_\mu(\alpha)&=\max\left\{h^\ba_{\mu}(T):\;
\lambda\in \M(X,T),\; ({\bf \Psi}_\mu)_*(\lambda)=-\alpha\right\}\\
&=\inf \left\{P^{\ba}(T, q{\bf \Psi}_\mu)+\alpha q:\; q\in \R\right\}.
\end{split}
\end{equation*}
\end{thm}
\begin{proof}
This result is just a corollary of Theorem~\ref{thm-0.7'}. Indeed, thanks to Theorem~\ref{thm-0.9'}(2) we can write
\begin{eqnarray*}
\frac{\log
\mu(B(x,e^{-n/a_1}))}{-n/a_1}&=&-a_1\frac{{\bf \Psi}^{(1)}_{\mu,n}(x)}{n}-a_1\sum_{i=2}^k \frac{{\bf \Psi}^{(i)}_{\mu,\ell_i(n)}(x)}{n}- \frac{{\bf \Psi}^{(i)}_{\mu,\ell_{i-1}(n)}(x)}{n}+o(1)\\
&=&-a_1\frac{{\bf \Psi}^{(1)}_{\mu,n}(x)}{n}-a_1\sum_{i=2}^k  \frac{b_i{\bf \Psi}^{(i)}_{\mu,\lfloor b_in\rfloor}(x)}{\lfloor b_in\rfloor }- \frac{b_{i-1}{\bf \Psi}^{(i)}_{\mu,\lfloor b_{i-1}n\rfloor }(x)}{\lfloor b_{i-1}n\rfloor }+o(1),
\end{eqnarray*}
with $b_i=(a_1+\cdots +a_i)/a_1$. Thus, any set $E_\mu(\alpha)$ takes the form $E_{\{{\bf \Phi}^{(j)}\},{\bc}}(\alpha)$, with $\sum_{j=1}^r {\bf \Phi}^{(j)}=-{\bf\Psi}$.
\end{proof}
\subsection*{More geometric applications}
A parallelepiped is a subset of $X$ of the form
$$
R(I_1,\dots,I_k)=\bigcap_{i=1}^k\tau_{i-1}^{-1}(I_i), \text{ with $I_i\in \bigcup_{n\ge 0}\A_i^n$}.
$$
If we fix $0\le\lambda_1\le \dots \le \lambda_k$ and set
$$
R_n(\lambda_1,\dots,\lambda_k,x)= R\Big(x_{|\lfloor\lambda_1 n\rfloor },\dots,  \tau_{i-1}(x_{|\lfloor\lambda_i n\rfloor }),\cdots,\tau_{k-1}(x_{|\lfloor\lambda_k n\rfloor })\Big),
$$
then
\begin{equation*}
\log \mu(R_n(\lambda_1,\dots,\lambda_k,x))=\sum_{i=1}^k {\bf \Psi}_{\mu,\lfloor\lambda_in\rfloor }^{(i)}(x)- {\bf \Psi}_{\mu,\lfloor\lambda_{i-1}n\rfloor }^{(i)}(x)+o(n),
\end{equation*}
with the convention $\lambda_0=0$. Consequently, Theorem~\ref{thm-0.7'} makes it also possible to compute the Hausdorff dimension of the sets
\begin{equation*}\label{rectangles}
\bigcap_{m=1}^M\left\{x\in X: \lim_{n\to \infty}\frac{\log
\mu\big  (R_n(\lambda^{(m)}_1,\dots,\lambda^{(m)}_k,x))}{-n}=\beta_m\right \},
\end{equation*}
where $\beta\in\R_+^M$ and each $(\lambda^{(m)}_i)_{1\le i\le m}$ satisfies $0\le\lambda^{(m)}_1\le \dots \le \lambda^{(m)}_k$.

\subsection{Moran measures}
Recall that the lower Hausdorff dimension of a Borel positive measure $\nu$ on $X$ is defined as $\underline{\dim}_H(\mu)=\inf\{\dim_H E: \nu(E)>0\}$. Equivalently, $\underline{\dim}_H(\mu)=\text{ess\,inf}_\nu\liminf_{r\to 0^+} \frac{\log(\nu(B(x,r)))}{\log (r)}$ (cf.  \cite{Fan94}). Recall also Remark \ref{re-2.2}. The main result in this subsection is the following.
\begin{thm}\label{Moran}
Let $(\mu_p)_{p\ge 1}\subset \M(X,T)$ be a sequence of invariant quasi-Bernoulli measures. Suppose that $(\mu_p)_{p\ge 1}$ converges in the weak-star topology to a measure $\mu$ and, moreover, $(h_{\mu_p\circ \tau_{i-1}^{-1}}(T_i))_{p\ge 1}$ converges to a limit $h_i$ for all $1\le i\le k$. Then there exists a probability measure $\nu$ of lower Hausdorff dimension larger than or equal to $\sum_{i=1}^ka_i h_i$ such that $\nu(\mathcal{G}(\mu))>0$. Consequently, $\dim_H\mathcal{G}(\mu)\geq  \sum_{i=1}^ka_i h_i$.
\end{thm}

\noindent
{\it Proof.} For each $p\ge 1$ and $1\le i\le k$ let us define
$\mu_{p,i}=\mu_p\circ\tau_{i-1}^{-1}$ and $\Psi^{(p)}_i:=\Psi^{\mu_p}_{i}=\left(\log \mu_{p,i}(\tau_{i-1}(x_{|n})\right)_{n=1}^\infty$. Notice that each $\Psi^{(p)}_{n,i}:=\log \mu_{p,i}(\tau_{i-1}(\cdot_{|n}))$ is locally constant over $n$-cylinders, and $h_{p,i}:=h_{\mu_{p,i}}(T_{i})=-(\Psi^{(p)}_{i})_*(\mu_p)$. Recall that as a part of our assumptions we have $\lim_{p\to\infty}h_{p,i}=h_i$ for each $1\le i\le k$.

Let $\widetilde \C$ be a countable set of additive potentials satisfying the bounded distortion property and such that for each $\Phi\in \C_{asa}(X,T)$ we can find a sequence $(\Phi^{(m)})_{m\ge 1}\subset \widetilde\C$ such that $\lim_{m\to\infty}\limsup_{n\to\infty} \|\Phi^{(m)}_n-\Phi_n\|_\infty/n=0$; the existence of such a set follows from Lemma~\ref{lem-2.1}(iii) and the separability of $C(X)$. For each $m,p\ge 1$ let $\alpha_{m,p}=\Phi^{(m)}_*(\mu_p)$. Since $\Phi^{(m)}_*(\cdot)$ is continuous over $\M(X,T)$ (cf. Lemma~\ref{lem-2.1}(ii)), and $\lim_{p\to \infty}\mu_p=\mu$,  we have  $\lim_{p\to\infty} \alpha_{m,p}=\Phi^{(m)}_*(\mu):=\alpha_m$.

For each $m\ge 1$, we denote as $c_m$ the constant associated with $\Phi^{(m)}$ in \eqref{BDP}.

The following proposition is a direct consequence of Kingman's sub-additive ergodic theorem applied for every $p\ge 1$ to each element of the families $\widetilde \C$ and $\{\Psi^{(p)}_i: 1\le i\le k\}$ and the ergodic measure $\mu_p$.
\begin{pro}For $p,N\in\mathbb{N}$ and $\varepsilon>0$, let
\begin{eqnarray*}
\mathcal{G}_1(p,N,\varepsilon)&=&\bigcap_{n\ge N} \bigcap_{i=1}^k\Big \{x\in X: \Big |\frac{\Psi^{(p)}_{n,i}(x)}{-n}-h_{p,i}\Big |\le \varepsilon\Big \},\\
\mathcal{G}_2(p,N,\varepsilon)&=&\bigcap_{n\ge N} \bigcap_{m=1}^p \Big \{x\in X: \Big |\frac{\Phi^{(m)}_{n}(x)}{n}-\alpha_{m,p}\Big |\le \varepsilon\Big \},
\end{eqnarray*}
and
$$
\mathcal{G}(p,N,\varepsilon)=\mathcal{G}_1(p,N,\varepsilon)\cap \mathcal{G}_2(p,N,\varepsilon).
$$
Then for all $p\in\mathbb{N}$ and $\varepsilon_p>0$, there exists an integer $N_p\ge 1$ such that
$$
\quad \mu_p(\G(p,N_p, \varepsilon_p))\ge 1-2^{-p}.
$$
\end{pro}
Let $(\varepsilon_p)_{p\ge 1}$ be a decreasing sequence converging to 0. With the notations in the previous proposition, for each $p$ we choose any $N'_p\ge N_p$. A precise choice of the integers $N'_p$ will be given later. Let $\mathcal{F}_p$ denote the $\sigma$-algebra generated by $\left\{[I]: I\in \A_1^{N'_p}\right\}$. We define
$$
\G_p=\Big \{I\in \A_1^{N'_p}: [I]\cap \G(p,N_p,\varepsilon_p)\neq\emptyset\Big \}.
$$
Then we denote by $\widetilde \mu_p$ the restriction of $\mu_p$ to $\mathcal{F}_p$ and define
\begin{equation*}
\begin{split}
\nu_p&=\otimes_{l=1}^p \widetilde\mu_l\text{ on $\Big(X,\otimes_{l=1}^p \mathcal{F}_p\Big )$}, \quad p\ge 1,\\
\displaystyle \nu&=\otimes_{p=1}^\infty \widetilde\mu_p\text{ on $\Big(X,\otimes_{p=1}^\infty \mathcal{F}_p\Big )$},
\end{split}
\end{equation*}
and
$$
\G:=\otimes_{p\ge 1}\G_p=\{I_1I_2\cdots I_p\cdots\in X_1: \forall\ p\ge 1, \ I_p\in  \G_p)\}.
$$
By construction, we have
$$
\nu(\G)=\prod_{p\ge 1} \widetilde \mu_p(\G_p)\ge \prod_{p\ge 1} \mu_p(\G(p,N_p,\varepsilon_p))=\prod_{p\ge 1}(1-2^{-p})>0.
$$
To conclude, it is enough to show that we can choose the sequence $(N'_p)_{p\ge 1}$ such that
\begin{eqnarray}
\label{gmu2}&\G\subset \G(\mu)\mbox{ and }\\
\label{dimnu2}& \displaystyle \liminf_{n\to\infty}\frac{\log\nu (B(x,e^{-n/a_1}))}{-n/a_1}\ge \sum_{i=1}^ka_ih_i \quad\mbox{ for all } x\in \G.
\end{eqnarray}
Then, $\nu:=\displaystyle \frac{\nu_{|\G}}{\nu(\G)}$ is desired.

\medskip

Let us establish \eqref{gmu2} and \eqref{dimnu2}.

\noindent{\it  Proof of \eqref{gmu2}.} We choose $N'_1=N_1$ and require that the sequence $(N'_p)_{p\ge 1}$ satisfies
\begin{equation}\label{controlmoran1}
M_p:=(p+1)\max_{1\le m\le p+1}\log (c_m)+\max_{1\le m\le p+1}\max_{1\le l\le N_{p+1}}\|\Phi^{(m)}_l\|_\infty=o\Big (\sum_{l=1}^{p}N'_l\Big)
\end{equation}
as $p\to\infty$. Then, for every $p\ge 1$ let
\begin{equation*}
\label{e-taa2}
L_p=\sum_{i=1}^pN'_i.
\end{equation*}

Due to the density of $\widetilde{\mathcal C}$, it is enough to prove that for each $m\ge 1$ and $x\in \G$ we have
\begin{equation}
\label{e-taa1}
\lim_{n\to\infty} \frac{\Phi^{(m)}_{n}(x)}{n}=\Phi^{(m)}_*(\mu) \;(:=\alpha_m).
\end{equation}
Fix $m\geq 1$ and $x\in \G$. For $n\ge N_1$, let $t(n)=\max\{p: L_p\le n\}$.  For all $n> L_{m+1}$, write
$$
\Phi^{(m)}_n(x)= \Phi^{(m)}_{L_m}(x)+\sum_{p=m+1}^{t(n)}\Phi^{(m)}_{N'_{p}}(\sigma^{L_{p-1}} x)+ \Phi^{(m)}_{n-L_{t(n)}}(\sigma^{L_{t(n)}} x).
$$
By construction, for $m+1\le p\le t(n)$ we have $\sigma^{L_{p-1}} x_{|N_{p}'}\in \G_{p}$. Consequently, there exists $x'\in\G(p,N_p,\varepsilon_p)$ such that $x'_{|N'_p}=\sigma^{L_{p-1}} x_{|N_p'}$ and thus
\begin{equation*}
|\Phi^{(m)}_{N'_p}(\sigma^{L_{p-1}} x)-N'_p\alpha_{m,p}|\le
|\Phi^{(m)}_{N'_p}(\sigma^{L_{p-1}} x)-\Phi^{(m)}_{N'_p}(x')|+ |\Phi^{(m)}_{N'_p}(x')-N'_p\alpha_{m,p}|\le \log (c_m)+ N'_p\varepsilon_p.
\end{equation*}
This yields
$$
\Big |\alpha_m\Big (\sum_{p=m+1}^{t(n)}N'_p\Big )-\Big (\sum_{p=m+1}^{t(n)}\Phi^{(m)}_{N'_p}(\sigma^{L_{p-1}} x)\Big )\Big |\le t(n) \log (c_m)+ \sum_{p=m+1}^{t(n)}N'_p(|\alpha_{m,p}-\alpha_m|+\varepsilon_p).
$$
Also, if $n-L_{t(n)}\le N_{t(n)+1}$, we have $ |\Phi^{(m)}_{n-L_{t(n)}}(\sigma^{L_{t(n)}} x)|\le \max_{1\le l\le N_{t(n)+1}}\|\Phi^{(m)}_l\|_\infty$, and if  $n-L_{t(n)}>N_{t(n)+1}$, then $[\sigma^{L_{t(n)}}x_{|n-L_{t(n)}}]\cap \G({t(n)+1},N_{t(n)+1},\varepsilon_{t(n)+1})\neq\emptyset$.  By the same argument as above we get
$$
\Big|\alpha_m(n-L_{t(n)})-\Phi^{(m)}_{n-L_{t(n)}}(\sigma^{L_{t(n)}} x)\Big|\le \log (c_m)+(n-L_{t(n)})(|\alpha_{m,t(n)+1}-\alpha_m|+\varepsilon_{t(n)+1}).
$$
It follows that
\begin{multline*}
 |\Phi^{(m)}_n(x)-n\alpha_m|\le |\Phi^{(m)}_{L_m}(x)|+M_{t(n)}\\
 + \Big (\sum_{p=m+1}^{t(n)}N'_p(|\alpha_{m,p}-\alpha_m|+\varepsilon_p)\Big )+(n-L_{t(n)})(|\alpha_{m,t(n)+1}-\alpha_m|+\varepsilon_{t(n)+1}).
 \end{multline*}
Due to our choice for $(N'_p)_{p\ge 1}$ and the fact that both $|\alpha_{m,p}-\alpha_m|$ and $\varepsilon_p$ tend to 0 as $p$ tends to $\infty$, as well as
 $M_{t(n)}=o(n)$, we obtain \eqref{e-taa1}.
This proves \eqref{gmu2}.

\medskip

\noindent
{\it Proof of \eqref{dimnu2}.} For each $p\ge 1$, since $\mu_p$ is quasi-Bernoulli, we can fix $\kappa_p>1$ such that \eqref{quasib} holds for $\mu=\mu_p$ and with the constant sequence $c=\kappa_p$.

We need additional properties for $(N'_p)_{p\ge 1}$.

The first one is that
$$
\displaystyle N'_{p+1}\ge \frac{a_1+\dots +a_k}{a_1} \Big (\sum_{i=1}^pN'_p\Big)= \frac{a_1+\dots +a_k}{a_1}L_p.
$$
The second one is
\begin{equation}\label{controlmoran2}
\sum_{l=1}^{p+2}\log (\kappa_{p})+\max_{1\le i\le k}\max_{j\in\{p+1,p+2\}}( h_iN_j+\max_{1\le n\le N_{j}}\|\Psi^{(j)}_{n, i}\|_\infty)=o(L_p)\text{ as $p\to\infty$}.
\end{equation}

 Fix $x=(x_i)_{i=1}^\infty\in \G$ and $n\ge N'_1$.  For $i=1,\ldots,k$, we use  $U_i$ to denote the word $x_{\ell_{i-1}(n)+1}\cdots x_{\ell_i(n)}$.
 Then by Lemma \ref{lem-simple},
  $$B(x,e^{-n/a_1})=\{y\in X: \;\forall\ 1\le i\le k, \ \tau_{i-1}(y)\in \tau_{i-1}(U_{1}\dots U_{i})\}.$$
  Write  $B=B(x,e^{-n/a_1})$
  for simplicity. Since $N'_{p+1}\ge (a_1+\dots+a_k) L_p/a_1$, there are only two cases to be distinguished: either $L_{t(n)}\le n <\ell_{k}(n)< L_{t(n)+1}$ or $L_{t(n)}\le n <L_{t(n)+1}=L_{t(\ell_{k}(n))}\le \ell_{k}(n)$. We deal with the second case and leave the easier first case to the reader.

Let $i_0$ be the unique $2\le i\le k$ such that $\ell_{i-1}(n)<L_{t(n)+1}\le \ell_{i}(n)$.
Let
$$
\mathcal{C}_n(B)= \left\{(J_1,\dots,J_k)\in \prod_{i=1}^{k}\A_1^{\ell_i(n)-\ell_{i-1}(n)}: \ \forall\ 1\le i\le k,\ \tau_{i-1}(J_i)=\tau_{i-1}(U_i)\right\}.
$$
We have
\begin{equation}\label{ballmass}
\nu(B)=\sum_{(J_1,\dots,J_k)\in \mathcal{C}_n(B)}\nu(J_1\cdots J_k).
\end{equation}
Write $J_1(=U_1)=\widetilde J_1\widehat J_1$  with $\widetilde J_1\in \A_1^{L_{t(n)}}$ and  $\widehat J_1\in \A_1^{n-L_{t(n)}}$, and write $J_{i_0}=\widetilde J_{i_0}\widehat J_{i_0}$, with $\widetilde J_{i_0}\in \A_1^{L_{t(n)+1}-\ell_{i_0-1}(n)}$ and $\widehat J_{i_0}\in \A_1^{\ell_{i_0}(n)-L_{t(n)+1}}$. This yields, by definition of $\nu_{t(n)}$ and $\nu$,
$$
\nu(B)=\sum_{(J_1,\ldots,J_k)\in \C_n(B)}\nu_{t(n)}(\widetilde J_1)\cdot  \mu_{t(n)+1}(\widehat J_1 J_2\cdots J_{i_0-1} \widetilde J_{i_0})
 \cdot  \mu_{t(n)+2}(\widehat J_{i_0} J_{i_0+1}\cdots J_k).
$$
Now, by using the quasi-Bernoulli properties of $\mu_{t(n)+1}$ and $\mu_{t(n)+2}$ we get
\begin{multline*}
 \nu(B)\approx \sum_{(J_1,\ldots,J_k)\in \mathcal{C}_n(B)}\nu_{t(n)}(\widetilde J_1)\cdot\mu_{t(n)+1}(\widehat J_1 )\cdot  \prod_{i=2}^{i_0-1}\mu_{t(n)+1}(J_i )\\
\cdot   \mu_{t(n)+1}(\widetilde J_{i_0}) \cdot \mu_{t(n)+2}(\widehat J_{i_0})\cdot
\prod_{i=i_0+1}^k \mu_{t(n)+2}(J_i),
\end{multline*}
where $\approx$ means the expressions  on its left and right hand sides differ from each other by a multiplicative constant belonging to $[\max (\kappa_{t(n)+1},\kappa_{t(n)+2})^{-k},\max (\kappa_{t(n)+1},\kappa_{t(n)+2})^k]$.

Accordingly, write $U_1=\widetilde U_1\widehat U_1$  with $\widetilde U_1\in \A_1^{L_{t(n)}}$ and  $\widehat U_1\in \A_1^{n-L_{t(n)}}$, and write $U_{i_0}=\widetilde U_{i_0}\widehat U_{i_0}$, with $\widetilde U_{i_0}\in \A_1^{L_{t(n)+1}-\ell_{i_0-1}(n)}$ and $\widehat U_{i_0}\in \A_1^{\ell_{i_0}(n)-L_{t(n)+1}}$.
Remembering the definition of $\mathcal{C}_n(B)$ we get
$$
\nu(B)\approx \prod_{i=1}^6T_i,
$$
where
\begin{equation*}
\begin{split}
T_1&=\nu_{t(n)}(\widetilde U_1),\ T_2=\mu_{t(n)+1}(\widehat U_1),\\
 T_3&= \prod_{i=2}^{i_0-1} \mu_{t(n)+1,i}(\tau_{i-1}(U_i)), \ T_4=\mu_{t(n)+1,i_0}(\tau_{i_0-1}(\widetilde U_{i_0})),\\
T_5&=\mu_{t(n)+2,i_0}(\tau_{i_0-1}(\widehat U_{i_0})) ,\ T_6= \prod_{i=i_0+1}^{k} \mu_{t(n)+2,i}(\tau_{i-1}(U_i)).
\end{split}
\end{equation*}
Let us write $\widetilde U_1= K_{1}\cdots K_{t(n)}$ with $K_{p}\in \A_1^{N'_p}$, for $1\le p\le t(n)$. By construction
$$
T_1=\prod_{p=1}^{t(n)}\mu_{p}(K_p).
$$
Now, we notice that  $x_{|\ell_{k}(n)}= K_1\cdots K_{t(n)} \widehat U_1U_2\cdots U_{i_0-1}\widetilde U_{i_0}\widehat U_{i_0} U_{i_0+1}
\cdots U_k$. Since $x\in\G$, we have $K_p$ belongs to $\G_p$ for $1\leq p\leq t(n)$. This  yields
\begin{eqnarray*}
\Big |\log T_1+h_1\Big (\sum_{p=1}^{t(n)}N'_p\Big )\Big|&=&\Big |\sum_{p=1}^{t(n)}\log \mu_p(K_{p})+h_1\Big (\sum_{p=1}^{t(n)}N'_p\Big )\Big|\\
&\le &R_1:=\sum_{p=1}^{t(n)}N'_p (|h_1-h_{p,1}|+\varepsilon_p).
\end{eqnarray*}

To control $T_2$, we notice that if $n-L_{t(n)}\le N_{t(n)+1}$ then $\widehat U_1\in \bigcup_{l=1}^{N_{t(n)+1}}\A_1^l$,  hence $|\log (T_2)|\le \max_{1\le l\le N_{t(n)+1}}\|\Psi_{l,1}^{(t(n)+1)}\|_\infty$, and $h_1(n-L_{t(n)})\le h_1N_{t(n)+1}$. If $N_{t(n)+1}<n-L_{t(n)} \le N'_{t(n)+1}$, since $[\widehat U_1]= [x_{L(n)+1}\cdots  x_{n}]$, we have  $[\widehat U_1]\cap \G(t(n)+1,N_{t(n)+1},\varepsilon_{t(n)+1})\neq\emptyset$, and since the mapping $\Psi^{(t(n)+1)}_{n-L_{t(n)},1}$ is constant over $[\widehat U_1]$ we obtain\begin{eqnarray*}
|\log T_2+h_1(n-L_{t(n)})|&\le&|\log T_2 +h_{t(n)+1,1}(n-L_{t(n)})|+(n-L_{t(n)})|h_1-h_{t(n)+1,1}|\\
&\le &  (n-L_{t(n)})(|h_1-h_{t(n)+1,1}|+\varepsilon_{t(n)+1}).
\end{eqnarray*}
In all cases,
\begin{eqnarray*}
|\log T_2+h_1(n-L_{t(n)})|&\le&R_2:=n(|h_1-h_{t(n)+1,1}|+\varepsilon_{t(n)+1})\\
&\mbox{}&\qquad
+h_1N_{t(n)+1}+\max_{1\le l\le N_{t(n)+1}}\|\Psi_{l,1}^{(t(n)+1)}\|_\infty.
\end{eqnarray*}

To control $T_3$ we proceed as follows. Fix $2\le i\le i_0-1$. Let $\overline U_i=[x_{L_{t(n)}+1}\cdots x_{\ell_{i-1}(n)}]$. By using the quasi-Bernoulli property of $\mu_{t(n)+1}$, which holds with the constant $\kappa_{t(n)+1}$, we can get
\begin{equation}\label{quasibmoran}
\Big |\log \mu_{t(n)+1}(U_i) - \big (\log \mu_{t(n)+1}(\overline U_iU_i)-\log \mu_{t(n)+1}(\overline U_i)\big )\Big |\le \log (\kappa_{t(n)+1}).
\end{equation}
Let $N\in\{\ell_{i-1}(n)- L_{t(n)}, \ell_{i}(n)-L_{t(n)}\}$, and set $U= \overline U_i$ if $N= \ell_{i-1}(n)- L_{t(n)}$ and $U=\overline U_iU_i$ otherwise.
If $N\le N_{t(n)+1}$, we have $|\log \mu_{t(n)+1,i}(\tau_{i-1}(U))|\le  \max_{1\le l\le N_{t(n)+1}}\|\Psi_{l,i}^{(t(n)+1)}\|_\infty$, and
$h_i N\le  h_iN_{t(n)+1}$. If $N_{t(n)+1}<N\le N'_{t(n)+1}$, since $[U]=[x_{L(n)+1}\cdots  x_{L(n)+N}]$, we have $[U]\cap \G(t(n)+1,N_{t(n)+1},\varepsilon_{t(n)+1})\neq\emptyset$, and since the mapping $\Psi^{(t(n)+1)}_{N,i}$ is constant over $[U]$ we obtain
\begin{eqnarray*}
&&|\log\mu_{t(n)+1,i}(\tau_{i-1}(U))+h_i N|\\
&&\quad \le |\log\mu_{t(n)+1,i}(\tau_{i-1}(U))+h_{t(n)+1,i}N|+N |h_i-h_{t(n)+1,i}|\\
&&\quad \le   N(|h_i-h_{t(n)+1,i}|+\varepsilon_{t(n)+1}),
\end{eqnarray*}
hence (using that $N\le \ell_i(n)$)
\begin{equation*}
\begin{split}
&\big |\log\mu_{t(n)+1,i}(\tau_{i-1}(\overline U_iU_i)- \log\mu_{t(n)+1,i}(\tau_{i-1}(\overline U_i))+h_i (\ell_i(n)-\ell_{i-1}(n))\big |\\
&\quad \le  r_i:=2 \big (\ell_i(n) (|h_i-h_{t(n)+1,i}|+\varepsilon_{t(n)+1})
+h_iN_{t(n)+1}+\max_{1\le l\le N_{t(n)+1}}\|\Psi_{l,i}^{(t(n)+1)}\|_\infty\big ).
\end{split}
\end{equation*}
Combining this with \eqref{quasibmoran} we get
\begin{eqnarray*}
\Big |\log T_3+\sum_{i=2}^{i_0-1}h_{i}(\ell_{i}(n)-\ell_{i-1}(n))\Big |\le R_3:=\sum_{i=2}^{i_0-1}\left( r_i+\log (\kappa_{t(n)+1})\right).
\end{eqnarray*}
By using the same arguments as for $T_2$ and $T_3$ we obtain
$$
\begin{cases}
|\log T_4+h_{i_0}(L_{t(n)+1}-\ell_{i_0-1}(n))|&\le R_4,\\
|\log T_5+h_{i_0}(\ell_{i_0}(n)-L_{t(n)+1})|&\le R_5, \\
\displaystyle \Big |\log T_6+\sum_{i=i_0+1}^{k}h_{i}(\ell_{i}(n)-\ell_{i-1}(n))\Big |&\le R_6,
\end{cases}
$$
with
\begin{eqnarray*}
R_4&=&2\big (L_{t(n)+1}(|h_{i_0}-h_{t(n)+1,i_0}|+\varepsilon_{t(n)+1})+h_{i_0+1}N_{t(n)+1}\\
&\mbox{ }&\quad+\max_{1\le l\le N_{t(n)+1}}\|\Psi_{l,i_0}^{(t(n)+1)}\|_\infty+\log (\kappa_{t(n)+1})\big );\\
R_5&=&\ell_{i_0}(n)(|h_{i_0}-h_{t(n)+2,i_0}|+\varepsilon_{t(n)+2})+h_{i_0}N_{t(n)+2}+\max_{1\le l\le N_{t(n)+2}}\|\Psi_{l,i_0}^{(t(n)+2)}\|_\infty;
\\
R_6&=&2 \sum_{i=i_0+1}^{k} \Big(\ell_{i}(n)(|h_{i}-h_{t(n)+2,i}|+\varepsilon_{t(n)+2})+ h_{i}N_{t(n)+2}\\
&&\mbox{ \hskip 1.5cm}+\max_{1\le l\le N_{t(n)+2}}\|\Psi_{l,i}^{(t(n)+2)}\|_\infty+ \log (\kappa_{t(n)+2})\Big).
\end{eqnarray*}
All the previous estimates yield, by construction of $(\varepsilon_p)_{p\ge 1}$, $(N'_p)_{p\ge 1}$ and the convergence of $h_{p,i}$ to $h_i$ as $p\to\infty$,
\begin{eqnarray*}
\left |\log (\nu(B))+\sum_{i=1}^{k}h_{i}(\ell_{i}(n)-\ell_{i-1}(n))\right|&\le& k\sum_{p=1}^{t(n)+2}\log (\kappa_p) +\sum_{i=1}^6 R_i \\
&=&o(L_{t(n)}+\ell_{k}(n))=o(n).
\end{eqnarray*}
Since $\displaystyle \lim_{n\to\infty}\frac{a_1}{n}\sum_{i=1}^{k}h_{i}(\ell_{i}(n)-\ell_{i-1}(n))=\displaystyle\sum_{i=1}^{k}a_ih_i$, we get $\displaystyle\lim_{n\to \infty}\frac{\log(\nu(B(x,e^{-n/a_1}))}{-n/a_1}=\sum_{i=1}^{k}a_ih_i$.
This finishes the proof of Theorem~\ref{Moran}. \eproof

\subsection{Proof of Theorem~\ref{thm-0.6'}}
 \label{s-5.4}
 By Theorem~\ref{thm-0.5'}, there exists a sequence of invariant quasi-Bernoulli measures  $(\mu_p)_{p\ge 1}$ converging to $\mu$ in the weak-star topology,  such that $h_{\mu_p\circ\tau_{i-1}^{-1}}(T_i)$ converges to $h_{\mu\circ \tau_{i-1}^{-1}}(T_i)$ for each $1\le i\le k$, as $p\to\infty$ (use the same argument as in Remark~\ref{simultapprox}). Then, the lower bound for $\dim_H \mathcal{G}(\mu)$ is a direct consequence of Theorem~\ref{Moran}. For the upper bound, we notice that $\mathcal{G}(\mu)\subset \bigcap_{\Phi\in \C(X,T)}E_\Phi(\Phi_*(\mu))$, where $\Phi\in \C(X,T)$ means $\Phi=(S_n\varphi)_{n=1}^\infty$ for some $\varphi\in C(X)$. Thus, by using Lemma~\ref{upperboundstep1} whose proof is independent of the present one, we obtain
\begin{equation*}
\begin{split}
\dim_H  \mathcal{G}(\mu)&\le \inf_{\Phi\in \C(X,T)} \dim_H E_\Phi(\Phi_*(\mu))\\
&\le \inf_{\Phi\in \C(X,T)}\inf _{q\in \R}P^{\ba}(T, q\Phi)-q\Phi_*(\mu)\\
&=\inf _{q\in \R}\inf_{\Phi\in \C(X,T)}P^{\ba}(T,   q\Phi)-q\Phi_*(\mu)\\
&=\inf_{\Phi\in \C(X,T)}P^{\ba}(T, \Phi)-\Phi_*(\mu).
\end{split}
\end{equation*}
 Now we note that, on the one hand, the $\ba$-weighted topological pressure is the Legendre-Fenchel  transform of the $\ba$-weighted entropy defined on the compact convex set $\M(X,T)$ of $C(X)^*$ endowed with the weak-star topology, and on the other hand, the $\ba$-weighted entropy is upper semi-continuous. Hence we have $\inf_{\Phi\in \C(X,T)}P^{\ba}(T, \Phi)-\Phi_*(\mu)=h^\ba_{\mu}(T)$ by mimicking the proof of Theorem 3.12 in \cite{Rue78}.  This yields the conclusion. \eproof

\subsection{Proof of Theorem~\ref{thm-0.7'}}
\label{s-5.5}
We first prove Theorem~\ref{thm-0.7'}(1).
 For $\alpha\in L_{\bf\Phi}$ let
 \begin{equation*}\label{fphi}
 f_{\bf \Phi}(\alpha)= \max\{h^\ba_{\mu}(T):\; \mu\in \M(X,T),\; {\bf \Phi}_*(\mu)=\alpha\}.
 \end{equation*}
 Since the mapping $\mu\in\mathcal{M}(X,T)\mapsto \sum_{i=1}^ka_ih_{\mu\circ\tau_{i-1}^{-1}}(T_i)$ is upper semi-continuous and affine, the equality $f_{\bf \Phi}(\alpha)= \inf \left\{P^{\ba}(T,  \bq\cdot {\bf \Phi})-\alpha\cdot \bq:\; \bq\in \R^d\right\}$ for $\alpha\in L_{\bf\Phi}$ is obtained by exactly the same arguments as those used to prove Theorem~5.2{\rm (iii)} in \cite{FeHu09a}; one just replaces the usual entropy by the $\ba$-weighted one.  Similarly, the proof of  the equivalence between {\rm (i)}, {\rm (iii)} and {\rm (iv)} follow the same lines as that of Theorem 5.2 {\rm (ii)} in \cite{FeHu09a}.

Consequently, to conclude it only remains to show that
\begin{eqnarray}
\label{check1}
&&E_{\{{\bf\Phi}^{(j)}\},\bc}(\alpha)\neq\emptyset\text{ and }\dim_H E_{\{{\bf\Phi}^{(j)}\},\bc}(\alpha)\ge \
f_{\bf \Phi}(\alpha) \text{ if $\alpha\in L_{\bf \Phi}$;}\\
&&\label{check2}\dim_HE_{\{{\bf\Phi}^{(j)}\},\bc}(\alpha)\le \inf \left\{P^{\ba}(T,  \bq\cdot {\bf \Phi})-\alpha\cdot \bq:\; \bq\in \R^d\right\} \text{ if $E_{\{{\bf\Phi}^{(j)}\},\bc}(\alpha)\neq\emptyset$},
\end{eqnarray}
since these properties clearly yield the equivalence of {\rm (i)} and {\rm (ii)}, as well as the value of $\dim_H E_{\{{\bf\Phi}^{(j)}\},\bc}(\alpha)$.

Assertion \eqref{check1} is an immediate consequence of Theorem~\ref{thm-0.6'} and the following lemma.
\begin{lem}\label{lem-5.12}
Let $\alpha=(\alpha_1,\dots,\alpha_d)\in L_{\bf\Phi}$ and $\mu\in\M(X,T)$ such that ${\bf\Phi}_*(\mu)=\alpha$. We have $\G(\mu)\subset E_{\{{\Phi}^{(j)}\},\bc}( \alpha)$.
\end{lem}
\noindent
\begin{proof}[Proof of Lemma \ref{lem-5.12}] By definition of ${\bf \Phi}$, we have $\alpha_i=\sum_{j=1}^r ({\Phi}^{(j)}_{i})_*(\mu)$ for each $1\le i\le d$. Moreover, by the definition of $\G(\mu)$,  we have $\G(\mu)\subset E_{{\Phi}^{(j)}_i}( ({\Phi}^{(j)}_{i})_*(\mu))$ for each $1\le j\le r$ and $1\le i\le d$, hence for each $x\in\G(\mu)$ we have $\lim_{n\to\infty}\sum_{j=1}^r \frac{{ \Phi}^{(j)}_{\lfloor c_jn\rfloor ,i}(x)}{\lfloor c_jn\rfloor }=\alpha_i$ for each $1\le i\le d$. This yields $\G(\mu)\subset E_{\{{\Phi}^{(j)}\},\bc}( \alpha)$. \end{proof}

\smallskip

Now we establish \eqref{check2}.  Define the following sequence of functions
\begin{equation}\label{Psichapeau}
{\bf \Phi}_{\bc,n}=n \sum_{j=1}^r \frac{{\bf \Phi}^{(j)}_{\lfloor c_jn\rfloor }}{\lfloor c_jn\rfloor }.
\end{equation}
We have the following lemma, which yields \eqref{check2}.
\begin{lem}\label{upperboundstep1}
Fix $\alpha\in\mathbb{R}^d$ and suppose that $E_{\{{\bf\Phi}^{(j)}\},\bc}(\alpha)\neq\emptyset$. For every $\varepsilon>0$ and $\bq\in\R^d$, we have $
\dim_H E_{\{{\Phi}^{(j)}\},\bc}( \alpha,\varepsilon)\le P^{\ba}(T,\bq\cdot{\bf \Phi})-\alpha\cdot\bq+(4|\bq|+a_1)\varepsilon$, where $E_{\{{\Phi}^{(j)}\},\bc}( \alpha,\varepsilon)=\{x\in X: \limsup_{n\to\infty}|{\bf \Phi}_{\bc,n}(x)/n-\alpha|\le \varepsilon\}$. Consequently, if $E_{\{{\bf\Phi}^{(j)}\},\bc}(\alpha)\neq\emptyset$, then $ \dim_H E_{\{{\Phi}^{(j)}\},\bc}( \alpha)\le \inf_{\bq\in\R^d} P^{\ba}(T,\bq\cdot{\bf \Phi})-\alpha\cdot\bq$, i.e.,  \eqref{check2} holds.
\end{lem}

\noindent
{\it Proof of Lemma \ref{upperboundstep1}}. Since $E_{\{{\bf\Phi}^{(j)}\},\bc}(\alpha)=E_{\{{\bf\Phi}^{(j)}\},\lambda\bc}(\alpha)$ for all $\lambda>0$, without loss of generality we assume that $c_j>1$ for all $j$.

Fix $\varepsilon>0$ and $\bq\in\R^d$. For each $1\le j\le r$, choose ${\bf \widetilde \Phi}^{(j)}\in \C_{aa}(X,T)^d$ such that each of its components satisfies the bounded distortion property and
$$\sup_{1\le i\le d}\limsup_{n\to\infty}\|{\bf \widetilde \Phi}^{(j)}_{i,n}-{\bf\Phi}^{(j)}_{i,n}\|_\infty/n\le \varepsilon/r.$$
 Then we define ${\bf\widetilde\Phi}=\sum_{j=1}^r{\bf \widetilde \Phi}^{(j)}$ and the sequence of functions
\begin{equation*}
{\bf \widetilde\Phi}_{\bc,n}=n \sum_{j=1}^r \frac{{\bf \widetilde \Phi}^{(j)}_{\lfloor c_jn\rfloor }}{\lfloor c_jn\rfloor }\quad (n\ge 1).
\end{equation*}
Endow the space $\R^d$ with the norm $|(z_1,\dots,z_d)|=\max_{1\le i\le d}|z_i|$. By construction we have $\limsup_{n\to\infty}\|{\bf \widetilde\Phi}_{\bc,n}-{\bf\Phi}_{\bc,n}\|_\infty/n\le \varepsilon$ so
\begin{equation*}\label{approx}
E_{\{{\bf\Phi}^{(j)}\},\bc}(\alpha,\varepsilon)\subset E_{\{{\bf\widetilde \Phi}^{(j)}\},\bc}( \alpha,2\varepsilon)=\{x\in X: \limsup_{n\to\infty}|{\bf \widetilde\Phi}_{\bc,n}(x)/n-\alpha|\le 2\varepsilon\}.
\end{equation*}
The definition of the $\ba$-weighted topological pressure implies
\begin{equation}\label{controlP}
|P^{\ba}(T, \bq\cdot{\bf\widetilde \Phi})-P^{\ba}(T, \bq\cdot{\bf\Phi})|\le |\bq|\varepsilon.
\end{equation}
Let us denote by $\mu_\bq$ the unique $\ba$-weighted equilibrium state of $\bq\cdot{\bf\widetilde\Phi}$ (see Theorem~\ref{thm-0.2'}).
The following key property holds.

\begin{lem}\label{McMullen}
For all $x\in X$, we have $\limsup_{n\to\infty}f_n(x)^{1/n}\ge 1$, where
$$
f_n(x)=\displaystyle \frac{\mu_{\bq}(B(x,e^{-n/a_1}))}{\exp\big ((\bq\cdot  {\bf \widetilde\Phi}_{\bc,n}(x)-nP^{\ba}(T,\bq\cdot{\bf \widetilde\Phi}))/{a_1}\big )}.
$$
\end{lem}
It is worth mentioning that the idea of considering the asymptotic behavior of such a function $f_n$ at {\it each} point of $X$ goes back to \cite{McM84} for the upper bound estimate of $\dim_H X$ when $k=2$. The proof of  Lemma \ref{McMullen} will be given later. To finish the proof of Lemma \ref{upperboundstep1}, we need the following classical lemma.
\begin{lem}[\cite{Bil65}, Ch. 14]\label{up}
Let $E$  be a non-empty subset of a compact metric space $(Y,d)$ endowed with an ultrametric distance. Let $\nu$ be a positive Borel measure on $Y$. Then $\dim_HE\le \sup_{x\in E} \displaystyle \liminf_{r\to 0^+}\frac{\log\nu(B(x,r))}{\log (r)}$.
\end{lem}

Now, if $x\in E_{\{{\bf\widetilde \Phi}^{(j)}\},\bc}( \alpha,2\varepsilon)$ then, due to Lemma~\ref{McMullen}, for infinitely many $n$ we have simultaneously $f_n(x) \ge \exp(-n\varepsilon)$, and $\exp(\bq\cdot  {\bf \widetilde\Phi}_{\bc,n}(x))\ge \exp (n\alpha\cdot \bq)-3|\bq|\varepsilon n$. Consequently,
\begin{equation*}\label{exponent}
\liminf_{n\to\infty}\frac{\log\mu_\bq(B(x,e^{-n/a_1}))}{-n/a_1}\le P^{\ba}(T,\bq\cdot{\bf \widetilde\Phi})-\alpha\cdot\bq+(3|\bq|+a_1)\varepsilon.
\end{equation*}
Now, Lemma~\ref{up} and \eqref{controlP} yield
\begin{equation*}
\begin{split}
\dim_H E_{\{{\bf\Phi}^{(j)}\},\bc}(\alpha,\varepsilon)&\le\dim_H  E_{\{{\bf\widetilde \Phi}^{(j)}\},\bc}( \alpha,2\varepsilon) \le
P^{\ba}(T,\bq\cdot{\bf \widetilde\Phi})-\alpha\cdot\bq+(3|\bq|+a_1)\varepsilon\\
&\le P^{\ba}(T,\bq\cdot{\bf \Phi})-\alpha\cdot\bq+(4|\bq|+a_1)\varepsilon.
\end{split}
\end{equation*}
Letting $\varepsilon\to 0$, we obtain $\dim_H E_{\{{\bf\Phi}^{(j)}\},\bc}(\alpha)\le P^{\ba}(T,\bq\cdot{\bf \Phi})-\alpha\cdot\bq$. Since $\bq\in \R^d$ is arbitrarily given, we have
$$
\dim_H E_{\{{\Phi}^{(j)}\},\bc}( \alpha)\le \inf_{\bq\in\R^d} P^{\ba}(T,\bq\cdot{\bf \Phi})-\alpha\cdot\bq.
$$
This finishes the proof of Lemma \ref{upperboundstep1}.
\eproof

\smallskip
Before we prove Lemma~\ref{McMullen}, we  give some auxiliary lemmas.
  \begin{lem}[\cite{KePe96}, Lemma 4.1] Let $m\ge 1$ be an integer.  For $1\le j\le m$ let $f_j:\N\to\R$, $\beta_j>0$ and $\lambda_j>0$. If $\sup_{n\ge 1}|f_{j}(n+1)-f_j(n)|<\infty$ for each $j$, then
$$
\limsup_{t\to\infty}\frac{1}{t}\sum_{j=1}^r \left (\beta_jf_j\Big(\Big\lfloor\frac{t}{\lambda_j}\Big\rfloor\Big )-f_j\Big(\Big\lfloor\frac{\beta_jt}{\lambda_j}\Big \rfloor\Big )\right )\ge 0.
$$
\end{lem}
The following lemma is essentially the same as the above one.
\begin{lem}\label{corlemKePe}
Let $m\ge 1$ be an integer. Consider $(\beta_p)_{1\le j\le m}$ and $(\gamma_p)_{1\le j\le m}$ two positive vectors, as well as $v_1,\dots,v_m$, $m$ bounded sequences such that $v_j(n+1)-v_j(n)=O(n^{-1})$ for each $1\le j\le m$. Then $\limsup_{n\to\infty} \sum_{j=1}^m v_j(\lfloor \beta_j n\rfloor )-v_j(\lfloor\gamma_jn\rfloor )\ge 0$.
\end{lem}
\noindent
{\it Proof of Lemma~\ref{McMullen}.} Let us denote $\bq\cdot{\bf \widetilde\Phi}_\bc$, $\bq\cdot{\bf\widetilde \Phi}$ and $\bq\cdot {\bf \widetilde \Phi}^{(j)}$ by $\widetilde\Phi_c$, $\widetilde \Phi$ and $\widetilde \Phi^{(j)}$ respectively. Next write $\widetilde\Phi_\bc$ under the following form:
$$
\widetilde\Phi_{\bc,n}=\widetilde\Phi_n+n\sum_{j=1}^r\frac{{\widetilde\Phi}^{(j)}_{\lfloor c_jn\rfloor }}{\lfloor c_jn\rfloor }-\frac{{\widetilde\Phi}^{(j)}_n}{n}.
$$

Let $x\in X$, $n\ge 1$ and let $B=B(x,e^{-n/a_1})$. By Lemma \ref{lem-gibbs}, we have
\begin{eqnarray*}
 \mu_\bq(B)&\approx&\exp \Big (\frac{-n P^\ba(T,\widetilde \Phi)}{a_1}\Big)\exp (\widetilde \Phi_n(x)/a_1) \prod_{j=1}^{k-1}\frac{\widetilde\phi^{(j)}(\tau_j(x_{|\ell_{j+1}(n)}))^{1/A_{j+1}}}{\widetilde\phi^{(j)}(\tau_j(x_{|\ell_{j}(n)}))^{1/A_{j}}}.
\end{eqnarray*}
Combining this with the definition of $f_n(x)$ yields
\begin{equation}\label{fn}
f_n(x)\approx\exp((\widetilde\Phi_n-\widetilde\Phi_{\bc,n}))(x)/a_1)
\prod_{j=1}^{k-1}\frac{\widetilde\phi^{(j)}(\tau_j(x_{|\ell_{j+1}(n)}))^{1/A_{j+1}}}{\widetilde\phi^{(j)}(\tau_j(x_{|\ell_{j}(n)}))^{1/A_{j}}}.
\end{equation}
Now, for $n\ge 1$, let us define
$$
\begin{cases}
u^{(j)}(n)=-\widetilde\Phi^{(j)}_n(x)/(a_1n) &\text{ for $1\le j\le r$,}\\
 \widetilde u^{(j)}(n)=\log \widetilde\phi^{(j)}(\tau_{j}(x_{|n}))/(a_1n)&\text{ for $1\le j\le k-1$}.
 \end{cases}
 $$
We notice that since the almost additive potentials $\widetilde \Phi$ and ${\bf\widetilde\Phi}^{(j)}$ satisfy the bounded distortion property, for any $v\in\{u^{(j)},\widetilde u^{(j)}: 1\le j\le r,\ 1\le i\le k-1\}$ the sequence $(v(n))_{n\ge 1}$ is bounded and $v(n+1)-v(n)=O(n^{-1})$. Then, by using \eqref{fn} we can get
$$
\frac{\log f_n(x)}{n}=\sum_{j=1}^r(u^{(j)}(\lfloor c_jn\rfloor )-u^{(j)}(n))+\sum_{j=1}^{k-1} (\widetilde u^{(i)}(\lfloor\widetilde c_i n\rfloor )-\widetilde u^{(i)}(\lfloor\widetilde c_{i-1} n\rfloor )) +O\Big(\frac{1}{n}\Big).
$$
Then, the fact that $\limsup_{n\to\infty} \frac{\log f_n(x)}{n}\ge 0$ comes from Lemma~\ref{corlemKePe}.
This finishes the proof of Lemma \ref{McMullen}.
\eproof

\medskip

\noindent Now we come to the proof of Theorem \ref{thm-0.7'}(2).  It is based on the following lemma and a modification of the Moran construction achieved in the proof of Theorem~\ref{Moran}. The proof of the lemma is postponed to the end of the section.

\begin{lem}\label{divergence}
Assume that $L_{\bf \Phi}$ is not a singleton. Then for all $\varepsilon>0$, there are two invariant quasi-Bernoulli measures $\nu_1$ and $\nu_2$ on $X$ with ${\bf\Phi}_*(\nu_1)\neq {\bf \Phi}_*(\nu_2)$, and  a non-negative vector $(h_i)_{1\le i\le k}$ such that $\sum_{i=1}^ka_ih_i\ge \dim_H X-\varepsilon$ and $h_{\nu_{l}\circ\tau_{i-1}^{-1}}(T_{i})\ge h_i$ for each $l\in \{1,2\}$ and $1\le i\le k$.
\end{lem}
Let $\delta=|{\bf\Phi}_*(\nu_1)-{\bf\Phi}_*(\nu_2)|$. For each $1\le j\le r$ and $1\le i\le d$, let $g^{(j)}_i$ be a H\"older potential such that $\limsup_{n\to\infty}\|\Phi^{(j)}_{i,n}-S_ng^{(j)}_i\|_\infty/n\le \delta/8 r$. For each $1\le j\le r$ let ${\bf G}^{(j)}=( (S_ng^{(j)}_i)_{n=1}^\infty)_{1\le i\le d}$, and define ${\bf G}=\sum_{j=1}^r {\bf G}^{(j)}$. By construction, we have $\limsup_{n\to\infty}\|{\bf \Phi}_{\bc,n}-{\bf G}_{\bc,n}\|_\infty/n\le \delta/8$ (recall that ${\bf \Phi}_{\bc,n}$ is defined as in \eqref{Psichapeau}, and we define ${\bf G}_{\bc,n}$ similarly). Moreover, for each $l\in\{1,2\}$ we have $|{\bf \Phi}_*(\nu_l)-{\bf G}_*(\nu_l)|\le \delta/8$, hence $|{\bf G}_*(\nu_1)-{\bf G}_*(\nu_2)|\ge 3\delta/4$. Thus, the set $$D_{\bf G}= \bigcap_{l=1}^2 \{x\in X: \liminf_{n\to\infty} | {\bf G}_{\bc,n}(x)-n{\bf G}_*(\nu_l)|/n\le \delta/4\}$$ is included in the set of divergent points $X\setminus \bigcup_{\alpha\in L_{\bf\Phi}} E_{\{{\bf\Phi}^{(j)}\},\bc}(\alpha)$, and the conclusion will follow if we prove that
$$
\dim_H D_{\bf G}\ge  \dim_H X-\varepsilon.
$$
Now we briefly explain how to modify the Moran construction done in the proof of Theorem~\ref{Moran}. At first, without loss of generality, we suppose that the $c_j$'s are greater than 1. Also, we include the potentials $(S_ng^{(j)}_i)_{n=1}^\infty)$ in the family $\widetilde C$. Then,  the only changes are that for each $p\ge 1$, one takes $\mu_{2p-1}=\nu_1$ and $\mu_{2p}=\nu_2$ and to the controls \eqref{controlmoran1} and \eqref{controlmoran2} one adds  $L_{p-1}=o(\sqrt{N'_p})$. Then, for $p\ge 1$, let $n_p=L_{p-1}+\sqrt{N'_p}$. For $p$ large enough, for each $1\le j\le r$ we have $\lfloor c_j n_{p}\rfloor \in [L_{p-1}+\sqrt{N'_p}, L_{p}]$, so that for each $x\in\G$, $1\le j\le r$ and $1\le i\le d$  we have $\lim_{p\to\infty} S_{\lfloor c_j n_{2p-1}\rfloor }g^{(j)}_i(x)/\lfloor c_j n_{2p-1}\rfloor =\nu_1(g^{(j)}_i)$ and $\lim_{p\to\infty} S_{\lfloor c_j n_{2p}\rfloor }g^{(j)}_i(x)/\lfloor c_j n_{2p}\rfloor =\nu_2(g^{(j)}_i)$. Consequently, for each $x\in \G$, we have $\lim_{p\to\infty}{\bf G}_{\bc,n_{2p-1}}(x)/n_{2p-1}={\bf G}_*(\nu_1)$ and $\lim_{p\to\infty}{\bf G}_{\bc,n_{2p}}(x)/n_{2p}={\bf G}_*(\nu_2)$, so $\G\subset D_{\bf G}$. Moreover, the simultaneous controls from below of the entropies $h_{\nu_{l}\circ\tau_{i-1}^{-1}}(T_{i})$ by the same $h_i$ yield, for every $x\in \G$, $\liminf_{n\to\infty}\frac{\log\nu(B(x, e^{-n/a_1}))}{-n/a_1}\ge \sum_{i=1}a_ih_i\ge \dim_H X-\varepsilon$. \eproof

\smallskip

\noindent
{\it Proof of Lemma~\ref{divergence}.}
Let $g\in C(X)$ be the zero function. Let $\nu_1$ be the $\ba$-weighted equilibrium state of $g$. Then by Theorem~\ref{thm-0.2'} and Remark~\ref{dimX},
$\nu_1$ is quasi-Bernoulli,  and $h^{\ba}_{\nu_1}(T):=\sum_{i=1}^kh_{\nu_1\circ \tau_{i-1}^{-1}}(T_i)=P^{\ba}(T,0)=\dim_HX$.


Fix $\varepsilon>0$, and for each $1\le i\le k$ let $h_i=h_{\nu_1\circ \tau_{i-1}^{-1}}(T_i)-\varepsilon/(a_1+\cdots+a_k)$. Since $L_{\bf \Phi}$ is not a singleton, we can pick $\mu\in \M(X,T)$ such that ${\bf \Phi}_*(\mu)\neq {\bf \Phi}_*(\nu_1)$. Take a large positive integer $n$ so that
$$
h_{\mu_2\circ \tau_{i-1}^{-1}}(T_i)\ge h_{\nu_1\circ \tau_{i-1}^{-1}}(T_i)-\varepsilon/(2a_1+\cdots+2a_k),\quad (1\le i\le k)
$$
where $\mu_2=(1-1/n)\nu_1+(1/n)\mu$. Note that
 ${\bf \Phi}_*(\mu_2)=(1-1/n){\bf \Phi}_*(\nu_1)+(1-1/n){\bf \Phi}_*(\mu)\neq {\bf \Phi}_*(\nu_1)$. Now by Remark \ref{simultapprox2}, we can pick an invariant quasi Bernoulli measure $\nu_2$ so that $h_{\nu_2\circ \tau_{i-1}^{-1}}(T_i)\ge h_{\mu_2\circ \tau_{i-1}^{-1}}(T_i)-\varepsilon/(2a_1+\cdots+2a_k)$, hence $h_{\nu_2\circ \tau_{i-1}^{-1}}(T_i)\ge h_i$ for each $1\le i\le k$. By construction, the pair of measures $\{\nu_1,\nu_2\}$ is as desired.
\eproof

\subsection{{\bf Proof of Theorem~\ref{weakGibbs}}}
\label{s-5.6}
Since $P^\ba(T,\Phi)/A_k$ is by construction equal to the classical topological pressure of $\Phi^\ba$, the problem reduces to proving the following assertion: Let $\Psi=(\log (\psi_n))_{n=1}^\infty \in \C_{asa}(X,T)$. There exists a fully supported measure $\nu$ such that
$$
\nu(x_{|n})\approx_n \exp(-nP(T,\Psi)) \psi_n(x)\quad (\forall\, x\in X, \forall \, n\ge 1).
$$

By Lemma~\ref{lem-2.1}(iii)), we can fix $(g_p )_{p\ge 1}$, a sequence of H\"older potentials such that $\limsup_{n\to\infty}\|(\log(\psi_n)-S_ng_p )\|_\infty/n\le 2^{-(p+1)}$ for each $p\ge 1$. Then fix a sequence $(r_p)_{p\ge 1}$ such that for each $p\ge 1$ we have $\sup_{n\ge r_p} \|(\log(\psi_n)-S_ng_p )\|_\infty/n\le 2^{-p}$. In particular, we have $|P_\psi-P_{g_p }|\le 2^{-p}$, where $P_\psi$ and $P_{g_p }$ stand for $P(T,\Psi)$ and $P(T,g_p )$ respectively.

For each $p\ge 1$, let $\mu_p$ be a Gibbs state for $g_p $ and $\kappa_p>1$ a constant such that
$$
\kappa_p^{-1}\le  \frac{\mu_p(x_{|n})}{\exp(-n P_{g_p } \exp (S_ng_p (x))} \le \kappa_p \quad (\forall\, x\in X, \forall \, n\ge 1).
$$
Let $(N_p)_{p\ge 1}$ be a sequence of integers such that
$$
\begin{cases}
N_p\ge \max (r_p,r_{p+1}),\\
(\log (\kappa_1)+\cdots +\log (\kappa_{p+1}))+L_{p-1}+M_{p+1}=o(\sqrt{N_p}),
\end{cases}
$$
where $L_{p}=\sum_{j=1}^{p}N_j$, and $M_p=\max\{\|g_j\|_\infty: 1\le j\le p\}$.

For each $p\ge 1$ let $\mathcal{F}_p$ denote the $\sigma$-algebra generated by $\left\{[I]: I\in \A_1^{N_p}\right\}$. Then denote by $\widetilde \mu_p$ the restriction of $\mu_p$ to $\mathcal{F}_p$ and define
\begin{equation*}
\begin{split}
 \nu&=\otimes_{p=1}^\infty \widetilde\mu_p\text{ on $\Big(X,\otimes_{p=1}^\infty \mathcal{F}_p\Big )$}.
\end{split}
\end{equation*}
For $n\ge N_1$ let $t(n)=\max\{p: L_p\le n\}$. For any $x\in X$ and $n\ge 1$ we have
$$
\nu(x_{|n})=\Big (\prod_{p=1}^{t(n)}\mu_p(T^{L_{p-1}}x_{|N_p})\Big ) \mu_{t(n)+1}\big (T^{L_{t(n)}}x_{|n-L_{t(n)}}\big ).
$$
For each $1\le p\le t(n)-1$ we have
\begin{eqnarray*}
&&\big |\log \big (\mu_p(T^{L_{p-1}}x_{|N_p})\big )- P_\psi N_p-S_{N_p}g_{t(n)+1}(T^{L_{p-1}}x_{|N_p})\big|\\
&\le & \log (\kappa_p) +|P_\psi-P_{g_p }|N_p+\big |S_{N_p}(g_p -g_{t(n)+1})(T^{L_{p-1}}x_{|N_p})\big|\\
&\le &  \log (\kappa_p) +2^{-p}N_p+2M_{t(n)+1} N_p.
\end{eqnarray*}
Moreover,
\begin{eqnarray*}
&&\big |\log \big (\mu_{t(n)}(T^{L_{t(n)-1}}x_{|N_{t(n)}})\big )- P_\psi N_{t(n)}-S_{N_{t(n)}}g_{t(n)+1}(T^{L_{{t(n)}-1}}x_{|N_{t(n)}})\big|\\
&\le & \log (\kappa_{t(n)}) +|P_\psi-P_{g_{t(n)}}|N_{t(n)}+\big |S_{N_{t(n)}}(g_{t(n)}-g_{t(n)+1})(T^{L_{t(n)-1}}x_{|N_{t(n)}})\big|\\
&\le &  \log (\kappa_{t(n)}) +2^{-{t(n)}}N_{t(n)}+\|S_{N_{t(n)}}(g_{t(n)}-g_{t(n)+1})\|_\infty.
\end{eqnarray*}
Also, denoting $n-L_{t(n)}$ by $R_n$ we have
\begin{eqnarray*}
&&
\big |\log \big (\mu_{t(n)+1}(T^{L_{t(n)}}x_{|R_n})\big )-P_\psi R_n -S_{R_n}g_{t(n)+1}(T^{L_{{t(n)}}}x_{|R_n})\big|\\
&\le&   \log (\kappa_{t(n)+1})+|P_\psi-P_{g_{t(n)+1}}| R_n\le \log (\kappa_{t(n)+1})+ 2^{-{t(n)+1}}R_n.
\end{eqnarray*}
We deduce from the previous estimations that
\begin{eqnarray*}
&&\big |\log (\nu(x_{|n}))-n P_\psi-\log (\psi_n(x))\big |\\
&\le& \|\log (\psi_n)-S_ng_{t(n)+1} \|_\infty+ \big |\log (\nu(x_{|n}))-n P_\psi-S_ng_{t(n)+1}(x)\big |\\
&\le& \|\log (\psi_n)-S_ng_{t(n)+1} \|_\infty+\|S_{N_{t(n)}}(g_{t(n)}-g_{t(n)+1})\|_\infty\\
&&+2M_{t(n)+1}L_{t(n)-1}+ 2^{-{t(n)+1}}(n-L_{t(n)})+ \sum_{p=1}^{t(n)} 2^{-p}N_p +\sum_{p=1}^{t(n)+1} \log (\kappa_p).
\end{eqnarray*}
Since $n\ge N_{t(n)}\ge \max (r_{t(n)},r_{t(n)+1})$ we have $\|\log (\psi_n)-S_ng_{t(n)+1} \|_\infty\le 2^{-t(n)+1}n$ and $\|S_{N_{t(n)}}(g_{t(n)}-g_{t(n)+1})\|_\infty \le (2^{-t(n)}+2^{-t(n)+1})N_{t(n)}$. So both terms are $o(n)$, uniformly in $x$. Moreover, by construction $2M_{t(n)+1}L_{t(n)-1}=(o(\sqrt n))^2=o(n)$, $2^{-{t(n)+1}}(n-L_{t(n)})+\sum_{p=1}^{t(n)} 2^{-p}N_p =o(n)$ and $\sum_{p=1}^{t(n)+1} \log (\kappa_p)=o(\sqrt n)$ uniformly in $x$. Consequently,
$$
\lim_{n\to\infty} \frac{1}{n}\sup_{x\in X} \big |\log (\nu(x_{|n}))-n P_\psi-\log (\psi_n(x))\big |=0.
$$

\noindent
When $\Phi\in C_{aa}(X,T)$ and satisfies the bounded distortion property, relation \eqref{massball1} is obtained by using \eqref{ballmass}, which holds for any positive measure $\nu$, and then the quasi-Bernoulli property of $\mu$. Then \eqref{massball1} yields  \eqref{logmassball1} in this case. To get \eqref{logmassball1} in the general case, let $(g_p)_{p\ge 1}$ as above. For each $p\ge 1$, let $\mu_p$ be the unique $\ba$-weighted-equilibrium state associated with $g_p$. By construction, we have $\lim_{p\to\infty}\limsup_{n\to\infty}\|{\bf \Psi}^{(i)}_{\mu, n}-{\bf \Psi}^{(i)}_{\mu_p,n}\|_\infty/n=0$ for each $1\le i\le k$; in particular ${\bf \Psi}^{(i)}_{\mu}\in \C_{asa}(X,T)$. Fix $\varepsilon >0$. Applying \eqref{ballmass} to $\mu$, we can find $p_\varepsilon\in\N_+$ and $N_\varepsilon\in N_+$ such that for $n\ge N_\varepsilon$ we have
$$
\begin{cases}
\|{\bf \Psi}^{(i)}_{\mu,n}-{\bf \Psi}^{(i)}_{\mu_{p_\varepsilon},n}\|_\infty\le n\varepsilon&\forall\, 1\le i\le k\\\exp (-2\ell_k(n)\varepsilon)\mu_{p_\varepsilon}(B)\le \mu (B)\le \exp (2\ell_k(n)\varepsilon)\mu_{p_\varepsilon}(B)&\forall\, B\in\mathcal{B}_n
\end{cases}.
$$
Let $c(x,n)$ be associated with $\mu_{p_\varepsilon}$ like in \eqref{logmassball1} for $\mu_{p_\varepsilon}$. We know that $c(x,n)$ is bounded independently of $x$ and $n$ by a constant $c(\mu_{p_\varepsilon})$. By using the validity of  \eqref{logmassball1} for $\mu_{p_\varepsilon}$, for every $n\ge N_\varepsilon$ large enough so that $c(\mu_{p_\varepsilon})\le n \varepsilon$, for every $x\in X$ and $B=B(x,e^{-n/a_1})$ we get
\begin{eqnarray*}
&&\Big |\log \mu(B)-{\bf \Psi}^{(1)}_{\mu,n}(x)+\sum_{i=2}^k {\bf \Psi}^{(i)}_{\mu,\ell_i(n)}(x)- {\bf \Psi}^{(i)}_{\mu,\ell_{i-1}(n)}(x)\Big|\\
&\le&  \Big |\log \mu_{p_\varepsilon}(B)-{\bf \Psi}^{(1)}_{\mu_{p_\varepsilon},n}(x)+\sum_{i=2}^k {\bf \Psi}^{(i)}_{\mu_{p_\varepsilon},\ell_i(n)}(x)- {\bf \Psi}^{(i)}_{\mu_{p_\varepsilon},\ell_{i-1}(n)}(x)\Big|\\
&&+ |\log \mu(B)-\log \mu_{p_\varepsilon}(B)|+2\sum_{i=1}^k\|{\bf \Psi}^{(i)}_{\mu,\ell_i(n)}-{\bf \Psi}^{(i)}_{\mu_{p_\varepsilon},\ell_i(n)}\|_\infty\\
& \le & c(\mu_{p_\varepsilon})+(2k+2) \ell_{k}(n)\varepsilon\le (n+(2k+2) \ell_{k}(n))\varepsilon.
\end{eqnarray*}
This yields the desired result.
 \eproof

\noindent{\bf Acknowledgements}. Both authors were partially supported by the   France/Hong Kong joint research scheme {\it PROCORE} (projects 20650VJ,  F-HK08/08T). Feng was also partially
supported by the RGC grant (project CUHK401008) in the Hong Kong Special Administrative
Region, China.

\end{document}